\documentclass[pdflatex]{sn-jnl}
\usepackage{amsfonts,amsmath,amssymb}
\usepackage{dsfont}
\usepackage{xcolor}
\usepackage{amsthm}
\usepackage{siunitx}
\usepackage[T1]{fontenc}
\usepackage{lmodern}

\newcommand{\CC}{\mathbb C}

\newtheorem{theorem}{Theorem}[section]
\newtheorem{definition}[theorem]{Definition}

\newtheorem{lemma}[theorem]{Lemma}

\newtheorem*{remark*}{Remark}
\newtheorem{example}[theorem]{Example}

\newcommand{\Prob}{\mathrm{Prob}}

\newcommand{\Ocal}[1]{\mathcal{O}(\epsilon^{#1})}

\DeclareMathOperator{\expct}{{\mathbb E}}

\usepackage[algo2e]{algorithm2e}
\usepackage{algorithmic}
\usepackage{algorithm, caption}

\newcommand{\blambda}{{\boldsymbol \lambda}}
\newcommand{\bI}{{\boldsymbol I}}

\begin{document}

\title{Randomized methods for computing joint eigenvalues, with applications 
to multiparameter eigenvalue problems and root finding}
\author[1]{Haoze He}\email{haoze.he@epfl.ch}
\author[1]{Daniel Kressner}\email{daniel.kressner@epfl.ch}
\author*[2]{Bor Plestenjak}\email{bor.plestenjak@fmf.uni-lj.si}

\affil[1]{\orgdiv{\'Ecole Polytechnique F\'ed\'erale de Lausanne (EPFL)}, \orgname{Institute of Mathematics}, \orgaddress{\city{1015 Lausanne}, \country{Switzerland}}}
\affil*[2]{\orgdiv{IMFM and Faculty of Mathematics and Physics}, \orgname{University of Ljubljana}, \orgaddress{\street{Jadranska 19}, \city{1000 Ljubljana}, \country{Slovenia}}}

\abstract{It is well known that a family of $n\times n$ commuting matrices 
can be simultaneously triangularized by a unitary similarity transformation.
The diagonal entries of the triangular matrices define the $n$ joint eigenvalues of the family. In this work, we consider the task of numerically computing approximations to such 
joint eigenvalues for a family of (nearly) commuting matrices.
This task arises, for example, in solvers for multiparameter eigenvalue
problems and systems of multivariate polynomials, which are our main motivations. 
We propose and analyze a simple approach that computes eigenvalues as
one-sided or two-sided Rayleigh quotients from eigenvectors of a random linear combination
of the matrices in the family. 
We provide some analysis and numerous numerical examples, showing that such randomized approaches can compute semisimple joint eigenvalues accurately and lead to improved performance of existing solvers.}



\pacs[Funding]{The work of the first and second author was supported by the Swiss National Science Foundation 
research project \emph{Probabilistic methods for joint and
 singular eigenvalue problems}, grant number: 200021L\_192049. The work of the third author was supported by the Slovenian Research and Innovation Agency (grants N1-0154 and P1-0194).}

\maketitle


\section{Introduction}\label{sec:intro}

Let ${\cal A}=\{A_1,\ldots,A_d\}$ be a commuting family of $n\times n$ complex matrices, i.e.,
$A_jA_k=A_kA_j$ for all $1\le j,k\le d$. It is well known that there
exists a unitary matrix $U$ such that all matrices $U^*A_1U,\ldots,U^*A_dU$ are upper triangular, 
see, e.g., \cite[Theorem 2.3.3]{horn13}. The $n$ $d$-tuples 
containing the diagonal elements of $U^*A_1U,\ldots,U^*A_dU$ are called the \emph{joint eigenvalues} of ${\cal A}$. 
In other words, if we partition $U=[u_1\ \cdots\ u_n]$, then 
\begin{equation}\label{eq:joint_eigenvalue_schur}
    \blambda^{(i)}=(\lambda_1^{(i)},\ldots,\lambda_d^{(i)})=(u_i^*A_1u_i,\ldots,u_i^*A_du_i), \qquad i=1,\ldots,n, 
\end{equation}
are the joint eigenvalues. For every joint eigenvalue $\blambda=(\lambda_1,\ldots,\lambda_d)$ of ${\cal A}$ 
there exists a nonzero 
\emph{common eigenvector} $x\in\CC^n$, such that $A_ix=\lambda_i x$ for $i=1,\ldots,d$.

The task of computing joint eigenvalues of a commuting family arises in a variety of applications. Our main motivation are 
numerical methods for multiparameter eigenvalue problems as well as multivariate root finding problems.

In this work, we will investigate solvers for joint eigenvalue problems that utilize a random linear combination 
\begin{equation}\label{eq:Amu}
A(\mu)=\mu_1 A_1 + \mu_2 A_2 +  \cdots + \mu_d A_d,
\end{equation}
where $\mu=[\mu_1\ \cdots\ \mu_d]^T$ is a random vector from the uniform distribution on the unit sphere in $\mathbb C^d$. Assuming that 
$A(\mu)$ can be diagonalized, we use its eigenvector matrix $X$ to diagonalize $A_1,\ldots,A_d$. More precisely, we extract
the joint eigenvalues $\blambda^{(i)}=(\lambda_1^{(i)},\ldots,\lambda_d^{(i)})$ from the diagonal elements of $X^{-1}A_1X,\ldots,X^{-1}A_dX$.

The idea of using a random linear combination like~\eqref{eq:Amu} is not new and has been used several times. For example, in~\cite{CorlessGianniTrager97, ManochaDemmel, multipareig_2023, VermeerschDeMoorLAA} the unitary transformation $U$ from the Schur decomposition $A(\mu)=U^*RU$ is used to simultaneously transform all matrices from ${\cal A}$ to
\emph{block} upper triangular form. Using the Schur decomposition, however, inevitably requires eigenvalue clustering, that is, the Schur decomposition of $A(\mu)$ needs to be reordered so that multiple eigenvalues are grouped together. Eigenvalue clustering is a numerically subtle task sensitive to roundoff error.
On the other hand, using eigenvectors for (partial) diagonalization does not require clustering and is thus simpler to use. Although unitary transformations appear to be numerically more appealing, in practice one often gets equally good or even better results by using eigenvectors for, e.g., solving multivariate root finding problems.

For commuting \emph{Hermitian} matrices $A_1,\ldots,A_d$, there is a unitary matrix $U$ that jointly diagonalizes all matrices. For this significantly simpler situation, compared to the general non-Hermitian case, randomized methods based on~\eqref{eq:Amu} have recently been analyzed in~\cite{HeKressner}, establishing favorable robustness and stability properties.

The rest of this work is organized as follows.
In Section~\ref{sec:preliminaries}, we provide some basic theory on commuting families and eigenvalue condition numbers. In Section~\ref{sec:Rayleigh} we present our main algorithm (Algorithm~\ref{alg:RJEA}) that uses a random linear combination to extract eigenvectors, followed by one- and two-sided Rayleigh quotients to extract joint eigenvalues. Section~\ref{sec:main} contains our main theoretical results, asymptotic error bounds of the approximations returned by Algorithm~\ref{alg:RJEA}. A probabilistic analysis is performed to account for randomness of the linear combinations and to better understand the effect of perturbations. Section~\ref{sec:synthetic} contains several numerical examples with synthetic data to understand different aspects of our proposed approach and the error bounds, highlighting that semisimple joint eigenvalues can be computed accurately  with high probability. In
Sections \ref{sec:mep} and \ref{sec:poly_roots} we present applications to multiparameter eigenvalue problems and solving systems of multivariate polynomials, respectively.


\section{Preliminaries}\label{sec:preliminaries}

\subsection{Basic properties of commuting families}\label{sec:commfam}

A family of matrices ${\cal A} =\{A_k \in \CC^{n \times n}\}_{k=1}^{d}$ is \emph{simultaneously diagonalizable} \cite[Definition 1.3.20]{horn13} if there exists an invertible matrix $X$ such that
all matrices $X^{-1}A_1X,\ldots,X^{-1}A_dX$ are diagonal. For a simultaneously diagonalizable family ${\cal A}$, it clearly holds that every $A_k$ is diagonalizable and $\mathcal{A}$ is a commuting family. It turns out that the converse also holds.
\begin{theorem}[{\cite[Theorem 1.3.21]{horn13}, \cite[Corollary 3, p. 224]{theoryofmatrices}}]\label{thm:sd_iff}
   A family of matrices ${\cal A}$ is simultaneously diagonalizable if and only if ${\cal A}$ is a commuting family and each $A_k$ is diagonalizable.
\end{theorem}

We recall from~\eqref{eq:joint_eigenvalue_schur} the definition of joint eigenvalues $\blambda^{(1)}, \ldots, \blambda^{(n)}$ for an $n\times n$ commuting family of $d$ matrices. Equivalently, $\blambda=(\lambda_1,\ldots,\lambda_d)$ is a joint eigenvalue if the eigenspaces
$\ker(A_k-\lambda_kI_n)$ have a nontrivial intersection. This observation leads to the following generalized notion of eigenspaces.

\begin{definition}[{\cite[Section 3.6]{AtkinsonBook}, \cite{kosirPlenstenjak02}}] \label{def:joint_eigenvalue_eigenspace}
Given a joint eigenvalue $\blambda=(\lambda_1,\ldots,\lambda_d) \in \CC^d$ 
of a commuting family $\cal A$, the \emph{common (right) eigenspace} belonging to $\blambda$ is defined as
\[\ker\left({\cal A}-\blambda \bI_n\right):=\bigcap_{k=1}^d \ker(A_k-\lambda_kI_n),\]
whereas the \emph{common \emph{left} eigenspace} belonging to $\blambda$ is 
defined as 
\[\ker\left(({\cal A}-\blambda \bI_n)^*\right):=\bigcap_{k=1}^d \ker\left((A_k-\lambda_kI_n)^*\right).\]
\end{definition}

The \emph{algebraic multiplicity} of $\blambda$ is defined as the number of times that $\blambda$ appears in the list of  $n$ joint eigenvalues $\blambda^{(1)},\ldots, \blambda^{(n)}$; see~\cite{kosirPlenstenjak02} and~\cite[Chapter 2]{invariantsubspacesbook} for an equivalent definition using \emph{root subspaces}. The \emph{geometric multiplicity} of $\blambda$ is defined as $\dim(\ker\left({\cal A}-\blambda \bI_n\right))$, the dimension of the common eigenspace. We say that $\blambda$ is \emph{simple} when its algebraic multiplicity is one and $\blambda$ is \emph{semisimple} if its algebraic and geometric multiplicities are equal.

Like for standard eigenvalue problems, the geometric multiplicity is bounded by the algebraic multiplicity. This fact can be derived from choosing a matrix $U$ that transforms all matrices 
to triangular form such that its first columns contain a basis of the common eigenspace $\ker({\cal A} - \blambda \bI_n)$. We will use a similar argument to show the following transformation, which will play an important role in our analysis in Section~\ref{sec:main}.

\begin{lemma}\label{lem:semisimple_partition}
Given a commuting family $\mathcal{A}$, let $\blambda = (\lambda_1,\ldots,\lambda_d)$ be a semisimple joint eigenvalue of multiplicity $p$. Then there exist invertible matrices $X= [X_1\ X_2] \in \mathbb{C}^{n \times n}$, $Y= [Y_1\ Y_2] \in  \mathbb{C}^{n \times n}$ such that
\begin{equation}\label{eq:semisimple_joint_normal_cond}
X_1^* X_1 = I_p, \qquad Y^*X = I_n, \qquad 
     Y^*A_kX = \begin{bmatrix}
       \lambda_k I_p & 0 \\
        0 & A^{(k)}_{22}
    \end{bmatrix},\quad k=1,\ldots,d,
\end{equation}
for some $A_{22}^{(k)}\in\mathbb{C}^{(n-p)\times (n-p)}$.
\end{lemma}
  To show Lemma~\ref{lem:semisimple_partition}, we need the following simple but important observation. 
  From the definition~\eqref{eq:joint_eigenvalue_schur} of the joint eigenvalues
$\blambda^{(i)}=(\lambda_1^{(i)},\ldots,\lambda_d^{(i)})$, $i=1,\ldots,n$, it follows that that the eigenvalues of the linear combination $A(\mu)$ from~\eqref{eq:Amu} are given by
\begin{equation*}
\lambda_i(\mu)=\mu_1\lambda_1^{(i)} + \mu_2\lambda_2^{(i)} + \cdots + \mu_d\lambda_d^{(i)},\quad i=1,\ldots,n.
\end{equation*}
We then have the following result from~\cite[Lemma 1]{HeKressner}.
\begin{lemma}\label{lem:lemma1HD}
    Let a commuting family ${\cal A}$, joint eigenvalues $\blambda^{(i)}$ and $\lambda_i(\mu)$ be defined as above. Then the following statement holds generically (with respect to $\mu \in \CC^d$) for all $i,j = 1,\ldots,d$: The equality $\lambda_i(\mu)=\lambda_j(\mu)$ holds if and only if $\blambda^{(i)}=\blambda^{(j)}$.
\end{lemma}

As a corollary of Lemma~\ref{lem:lemma1HD} it follows for a simultaneously diagonalizable family ${\cal A}$ that $\ker\left({\cal A}-\blambda \bI_n\right) = \ker\left(A(\mu) -\lambda(\mu) I_n\right)$ holds for generic $\mu \in \CC^d$, i.e., generic linear combinations preserve common eigenspaces. We can now prove Lemma~\ref{lem:semisimple_partition}.

\begin{proof}[Proof of Lemma \ref{lem:semisimple_partition}]
 Let the columns of $U_1 = [u_1 \ \cdots \ u_p] \in \CC^{n \times p}$ be an orthonormal basis of the common eigenspace $\ker\left({\cal A}-\blambda \bI_n\right)$, which, by assumption, has dimension $p$. Setting a unitary $U = [U_1 \ U_2] \in \CC^{n \times n}$ yields 
\[T_k = U^*A_kU = \begin{bmatrix}
    \lambda_k I_p & A^{(k)}_{12} \\
    0 & A^{(k)}_{22}
\end{bmatrix}, \quad k=1,\ldots,d,  \quad T(\mu) = U^*A(\mu)U = \begin{bmatrix}
    \lambda(\mu) I_p & A_{12}(\mu) \\
    0 & A_{22}(\mu)
\end{bmatrix}, \]
for every $\mu \in \CC^d$. Because $\blambda$ has multiplicity $p$, it cannot be a joint eigenvalue for the commuting family $A^{(1)}_{22},\ldots, A^{(d)}_{22}$. By Lemma~\ref{lem:lemma1HD}, there exists $\mu \in \CC^d$ such that $\lambda(\mu)$ is not an eigenvalue of $A_{22}(\mu)$. Standard arguments \cite[Theorem 1.5, Chapter V]{StewartSun} imply that there exists $R \in \CC^{p \times (n-p)}$
such that
\[
\overline T(\mu) :=\overline{Y}^* T(\mu) \overline{X} = \begin{bmatrix}
    \lambda(\mu)I_p & 0 \\
    0& A_{22}(\mu)
\end{bmatrix}\quad\textrm{for}\quad
\overline{X} = \begin{bmatrix}
    I_p & R \\
     0 & I_{n-p}
\end{bmatrix},\quad  \overline{Y} = \begin{bmatrix}
    I_p & 0 \\
   -R^* & I_{n-p}
\end{bmatrix}. \] 
Note that $\overline Y^* \overline X = I_n$. Because $\overline T(\mu)$ commutes with every $\overline{Y}^* T_k \overline{X}$ and $\lambda(\mu)$ is not an eigenvalue of $A_{22}(\mu)$, it follows that
\[\overline{Y}^* T_k \overline{X} = \begin{bmatrix}
    \lambda_k I_p & 0\\
    0& A^{(k)}_{22}
\end{bmatrix},\quad k = 1,\ldots d.\]
The proof is completed by setting $X =  U\overline X  $ and $Y = U\overline{Y}$.
\end{proof}

A consequence of Lemma~\ref{lem:semisimple_partition} is the following result that gives another equivalent characterization of simultaneous diagonalizability (in addition to Theorem~\ref{thm:sd_iff}).
\begin{theorem}\label{thm:joint_semisimple_iff}
   A commuting family ${\cal A}=\{A_k \in \CC^{n \times n}\}_{k=1}^{d}$ is simultaneously diagonalizable if and only if all joint eigenvalues of $\cal A$ are semisimple.
\end{theorem}
\begin{proof}
Let $X$ be an invertible matrix such that each $D_k = X^{-1}A_kX$ is diagonal. Since similarity transformations preserve joint eigenvalues, the semisimplicity of every joint eigenvalue $\blambda$ follows from $\dim(\ker(\mathcal{D} - \blambda \bI_n)) = \dim(\ker(\mathcal{A} - \blambda  \bI_n))$, where ${\cal D}=\{D_1,\ldots,D_p\}$.

If all joint eigenvalues of $\cal A$ are semisimple, the same property holds for the diagonal blocks $\{A_{22}^{(k)}\}_{k=1}^{d}$ in~\eqref{eq:semisimple_joint_normal_cond}. In turn, the proof in the other direction follows from the repeated application of Lemma~\ref{lem:semisimple_partition}.
\end{proof}

We now relate the diagonalizability of the family to the diagonalizability of a linear combination.

\begin{theorem}\label{thm:theorem2HD}
    Let ${\cal A}=\{A_k \in \CC^{n \times n}\}_{k=1}^{d}$ be a simultaneously diagonalizable commuting family. Then the following statement holds generically (with respect to $\mu \in \CC^d$): If an invertible matrix $X$ diagonalizes $A(\mu)$, then 
    $X$ diagonalizes $A_k$ for every $k=1,\ldots,d$.
\end{theorem}
\begin{proof}
The proof proceeds along the lines of the proof of~\cite[Theorem 2]{HeKressner}.
Without loss of generality, we may assume that the first $p$ columns of $X$ span the right eigenspace belonging
  to an eigenvalue $\lambda_1(\mu)$ of $A(\mu)$. Then 
  $$X^{-1}A(\mu)X=\left[\begin{matrix}\lambda_1(\mu)I_{p} & 0 \cr 0 & A_{22}(\mu)\end{matrix}\right]\quad
  \textrm{and}\quad
    X^{-1}A_kX=\left[\begin{matrix}A_{11}^{(k)} & A_{12}^{(k)} \cr A_{21}^{(k)} & A_{22}^{(k)}\end{matrix}\right],\quad k=1,\ldots,d.
  $$ 
  Since $A(\mu)$ and $A_k$ commute, we have
  $A_{12}^{(k)}(A_{22}(\mu)-\lambda_1(\mu)I_{n-p})=0$, $(A_{22}(\mu)-\lambda_1(\mu)I_{n-p})A_{21}^{(k)}=0$, and $A_{22}(\mu)A_{22}^{(k)}=A_{22}^{(k)}A_{22}(\mu)$.
  Since $A_{22}(\mu)-\lambda_1(\mu)I_{n-p}$ is invertible, $A_{12}^{(k)}=0$ and $A_{21}^{(k)}=0$. It follows from Lemma \ref{lem:lemma1HD} that $A_{11}^{(k)}$ is a multiple of the identity matrix for a generic $\mu$. 
  We continue the proof by considering $A_{22}(\mu)$ and $A_{22}^{(k)}$ in the same way.
\end{proof}

If all joint eigenvalues of ${\cal A}$ are simple, then
the common eigenvectors 
are uniquely determined (up to scaling). Therefore, for 
generic $\mu \in \CC^d$, we get essentially the same
eigenvector matrix $X$ by diagonalizing $A(\mu)$. 

If the diagonalizability assumption on ${\cal A}$ in Theorem~\ref{thm:theorem2HD}
is not satisfied, then, by Theorem~\ref{thm:sd_iff}, ${\cal A}$ contains at least one nondiagonalizable matrix. It can be shown that this implies that 
$A(\mu)$ is nondiagonalizable as well for almost every $\mu$.
 Note that the canonical structure of a general commuting family ${\cal A}$ can be quite complex, see, e.g.,~\cite{KosirLAA93, invariantsubspacesbook, commutativematrices}, and the matrices $A_k$ do not necessarily have the same Jordan  structure.
We will touch on this case again in Section~\ref{sec:Rayleigh}.

\subsection{Eigenvalue condition numbers}\label{sec:eigenvalue_condition_number}

Errors, due to roundoff or other sources, inevitably destroy commutativity properties. Hence, in practical situations one usually deals with a perturbed, \emph{nearly} commuting family of matrices $\mathcal{\widetilde A} = \{\widetilde A_1,\ldots,\widetilde A_d\}$, with $\widetilde A_i=A_i+\epsilon E_i$ for $i=1,\ldots,d$. To study the impact of these perturbations on eigenvalues, we first recall the notion of eigenvalue condition numbers.

For a simple eigenvalue $\lambda$ of a matrix $A  \in \mathbb{C}^{n \times n}$ it is well known \cite{Golub_VanLoan_4,firstOrderPerturation,horn13} that $\lambda$ is differentiable with respect to $A$ and that the corresponding eigenvalue $\widetilde \lambda(\epsilon)$ of the perturbed matrix $A + \epsilon E$ admits the expansion 
\begin{equation}\label{eq:simple_eigenvalue_expansion}
    \widetilde \lambda  = \lambda + (y^*Ex)\epsilon + \mathcal{O}(\epsilon^2),
\end{equation}
where $x$ and $y$ are right/left eigenvectors belonging to $\lambda$ with the normalization chosen such that $y^*x = 1$ and $\|x\|_2 = 1$. Here and in the following, $\|\cdot\|_2$ denotes the $2$-norm of a vector or matrix.
Motivated by~\eqref{eq:simple_eigenvalue_expansion}, the condition number of $\lambda$ is defined as follows.
\begin{definition}\label{def:eigenvalue_condition_simple}
With the notation introduced above, $\kappa(\lambda) = \|y\|_2$ is called
the eigenvalue condition number of a simple eigenvalue $\lambda$.
\end{definition}

The matrix $P_\lambda = xy^*$ is the spectral projection~\cite{Stewart73,Kato95,LAPACK} corresponding to the simple eigenvalue $\lambda$. Noting that $\kappa(\lambda) = \|P_\lambda\|_2$ motivates the following extension to semisimple eigenvalues.
\begin{definition} \label{def:eigenvalue_condition_semisimple}
The \emph{spectral projection} corresponding to a semisimple eigenvalue $\lambda$ of a matrix $A$ is the matrix $P_\lambda$ such that $P_\lambda$ is a projection, the column space of $P_\lambda$ is the right eigenspace $\ker(A-\lambda I_n)$ and the row space of  $P_\lambda$ is the left eigenspace $\ker((A-\lambda I_n)^*)$.
The eigenvalue condition number of $\lambda$ is defined as
    $\kappa(\lambda) = \|P_\lambda\|_2$.
\end{definition}
Note that LAPACK~\cite{LAPACK} also uses norms of spectral projections to define condition numbers for semisimple eigenvalues and clusters of eigenvalues.

Similar to~\eqref{eq:semisimple_joint_normal_cond}, if $\lambda$ is a semisimple eigenvalue of multiplicity $p$ of 
 a matrix $A  \in \mathbb{C}^{n \times n}$, then there exist invertible matrices $X = [X_1 \ X_2] \in \CC^{n \times n}, Y = [Y_1 \ Y_2] \in \CC^{n \times n}$ satisfying
 \begin{equation*}
  X_1^*X_1 = I_p, \quad Y^*X = I_n, \quad  Y^*AX = \begin{bmatrix}
     \lambda I_p & 0 \\
     0 & A_{22}
 \end{bmatrix}
 \end{equation*}
 for some $A_{22} \in \CC^{(n-p) \times (n-p)}$. It follows that $X_1 Y_1^*$ is the spectral projection corresponding to $\lambda$ and, hence, $\kappa(\lambda) = \|Y_1\|_2$.

The following result extends the expansion~\eqref{eq:simple_eigenvalue_expansion} to semisimple eigenvalues.

\begin{theorem}\label{thm:semisimple_eigenvalue_expansion}
    Given $A\in\CC^{n\times n}$, let $\lambda$ be a semisimple eigenvalue of multiplicity $p$ 
    and let $X_1$ and $Y_1$ be bases of $\ker(A-\lambda I_n)$ and $\ker((A-\lambda I_n)^*)$, respectively, such that $Y_1^*X_1 = I_p$. Then $A+\epsilon E$ has $p$ eigenvalues
    satisfying the expansion
    \[\widetilde \lambda_i = \lambda_i + \eta_i \epsilon + {o}(\epsilon),\quad i=1,\ldots,p,\]
    where $\eta_1,\ldots,\eta_p$ are the eigenvalues of $Y_1^*EX_1$.
    \end{theorem}

\begin{proof}
 This is a particular case covered by Theorem 6 in \cite{LaMaZh03_PertSemisimple} on
 eigenvalues of analytic matrix functions of the form $L(\lambda,\epsilon)$, where we take 
 $L(\lambda,\epsilon)=A-\lambda I + \epsilon E$.
\end{proof}

When imposing the normalization $X_1^*X_1 = I_p$, it follows that 
$|\eta_i| \le \|Y_1^*EX_1\|_2 \le \|Y_1\|_2 \|E\|_2$, which provides another justification for defining $\|Y_1\|_2$ to be the eigenvalue condition number.

By extending the notion of spectral projections to a commuting family $\cal A$, an extension of the 
condition number to semisimple eigenvalues follows naturally.
\begin{definition} \label{def:eigenvalue_condition_joint_semisimple}
The \emph{spectral projection} corresponding to a semisimple joint eigenvalue $\blambda$ of a commuting family $\cal A$ is the matrix $P_\blambda$ such that $P_\blambda$ is a projection, 
the column space of $P_\blambda$ is the right common eigenspace $\ker(\mathcal{A}-\blambda \bI_n)$ and the row space of  $P_\blambda$ is the left common eigenspace $\ker((\mathcal{A}-\blambda \bI_n)^*)$ (see Definition~\ref{def:joint_eigenvalue_eigenspace}).
The eigenvalue condition number of $\blambda$ is defined as
    $\kappa(\blambda) = \|P_\blambda\|_2$.
\end{definition}
The block diagonal decomposition~\eqref{eq:semisimple_joint_normal_cond} implies again that $\kappa(\blambda) = \|Y_1\|_2$ .

\section{Rayleigh quotients}\label{sec:Rayleigh}

Algorithm \ref{alg:RJEA} summarizes our simple approach for approximating joint eigenvalues of nearly commuting family. The two types of Rayleigh quotients used for this purpose are discussed in more detail below.

\begin{algorithm}[H]
    \caption{\textbf{R}andomized \textbf{J}oint \textbf{E}igenvalue \textbf{A}pproximation }
    \textbf{Input:} \text{A nearly commuting family $\widetilde{\mathcal{A}}= \{\widetilde A_1,\ldots, \widetilde A_d\}$, $\text{opt} \in \{\text{RQ1},\text{RQ2}\}$.}\\
     \textbf{Output:} \text{Approximations of joint eigenvalues of $\widetilde{\mathcal{A}}$.}
    \begin{algorithmic}[1]
    \label{alg:RJEA}
        \STATE Draw $\mu \in \mathbb{C}^d$ from the uniform distribution on the unit sphere.
        \STATE Compute $\widetilde A(\mu) = \mu_1 \widetilde A_1 + \cdots + \mu_d \widetilde A_d$.
        \STATE Compute invertible  matrices $\widetilde X, \widetilde Y$ such that the columns of $\widetilde X$ have norm $1$,\\ $\widetilde Y^* \widetilde X = I_n$,  and $\widetilde Y^* \widetilde A(\mu) \widetilde X$ is diagonal.
        \IF{opt $=$ RQ$1$}
        \RETURN $\widetilde{\blambda}_{ \mathrm{RQ1}}^{(i)} = (\widetilde x_i^* \widetilde A_1 \widetilde x_i, \ldots, \widetilde x_i^* \widetilde A_d \widetilde x_i), \quad i = 1,\ldots,n$.
      \ELSIF{opt $=$ RQ$2$}
        \RETURN $\widetilde{\blambda}_{ \mathrm{RQ2}}^{(i)} = (\widetilde y_i^* \widetilde A_1 \widetilde x_i, \ldots, \widetilde y_i^* \widetilde A_d \widetilde x_i), \quad i = 1,\ldots,n$.
      \ENDIF
    \end{algorithmic}
    \end{algorithm}

    The random vector $\mu \in \CC^d$ in Line $1$ of Algorithm~\ref{alg:RJEA} can be generated by drawing a complex Gaussian random vector $g$ (that is, the real and imaginary parts of its entries are i.i.d $\mathcal{N}(0,1/2)$) and setting $\mu = g / \|g\|_2$. Note that the matrix $\widetilde Y = \widetilde X^{-*}$ contains the left eigenvectors of $\widetilde A(\mu)$. The normalization conditions in Line $3$ correspond to what is computed by standard eigenvalue solvers, such as {\tt CGEEVX} in LAPACK~\cite{LAPACK}. 
    In the following, we provide details on the key quantities $\widetilde{\blambda}_{\text{RQ$1$}}^{(i)}$ and  $\widetilde{\blambda}_{\text{RQ$2$}}^{(i)}$ of Algorithm~\ref{alg:RJEA}.
\smallskip

\noindent\textbf{One-sided Rayleigh quotient}\quad  From each normalized common (right) eigenvector $x_i$ of $\cal A$, we can compute the \emph{(one-sided) Rayleigh quotient}
\begin{equation}\label{eq:1sRQ}
    \blambda^{(i)}_{\mathrm{RQ1}}(x_i,{\cal A}):=(x_i^*A_1x_i, \ldots, x_i^*A_dx_i).
\end{equation}
In the absence of noise, the quantity $\blambda^{(i)}_{\mathrm{RQ1}}(x_i,{\cal A})$ computed in Line $5$ of Algorithm~\ref{alg:RJEA} equals, by Theorem~\ref{thm:theorem2HD}, the joint eigenvalue $\blambda^{(i)}$ of $\cal A$ with probability $1$. In the presence of errors and noise, if $\widetilde x_i=x_i+ \epsilon \Delta x_i $ is an approximation to the exact common eigenvector $x_i$ and we apply it to $\widetilde A_k=A_k+ \epsilon E_k$, then for each component $\lambda_k^{(i)}$ of $\blambda^{(i)}$,
\[
\widetilde\blambda^{(i)}_{\mathrm{RQ1}}(\widetilde x_i,\widetilde{\cal A})_k-\lambda_k^{(i)}= \widetilde x_i^*(A_k-\lambda_k^{(i)}I)\widetilde x_i
= \left(\widetilde x_i^*(A_k-\lambda_k^{(i)}I)\Delta x_i + \widetilde x_i^*E_k\widetilde x_i\right)\epsilon.
\]
A rough error estimate is
\begin{equation}\label{eq:err_1sRQ}
|\widetilde\blambda^{(i)}_{\mathrm{RQ1}}(\widetilde x_i,\widetilde{\cal A})_k-\lambda_k^{(i)}|\le  \left(\|(A_k-\lambda_k^{(i)}I)\Delta x_i\|_2 + \|E_k\|_2\right)\epsilon.
\end{equation}
In Section \ref{sec:main}, we will derive a more refined bound.
\smallskip

\noindent\textbf{Two-sided Rayleigh quotient}\quad From each pair of left and right common eigenvectors $y_i,x_i$ of $\cal A$, we can compute the \emph{two-sided Rayleigh quotient}
\begin{equation}\label{eq:2SRQ}
\blambda_{\mathrm{RQ2}}^{(i)}(x_i,y_i,{\cal A}):=\left(\frac{y_i^*A_1 x_i}{y_i^* x_i}, \ldots, \frac{y_i^*A_d x_i}{y_i^* x_i}\right) = \left( y_i^*A_1 x_i, \ldots,y_i^*A_d x_i \right),
\end{equation}
using our normalization assumptions on $X$ and $Y$.
In the absence of noise, the quantity $\blambda^{(i)}_{\mathrm{RQ2}}(x_i,y_i,{\cal A})$  computed in Line $7$ of Algorithm~\ref{alg:RJEA} is again the joint eigenvalue $\blambda^{(i)}$ of $\cal A$ with probability $1$. 

If $\widetilde x_i=x_i+ \epsilon \Delta x_i $, $\widetilde y_i=y_i+ \epsilon \Delta y_i$ are approximations to the right/left common eigenvectors and $\widetilde A_k=A_k+ \epsilon  E_k$, then for each component $\lambda_k^{(i)}$ of $\blambda^{(i)}$,
\[
\widetilde\blambda^{(i)}_{\mathrm{RQ2}}(\widetilde x_i, \widetilde y_i, \widetilde{\cal A})_k-\lambda_k^{(i)}= \frac{\epsilon}{\widetilde y_i^*\widetilde x_i}(\widetilde y_i^*E_k\widetilde x_i) + \frac{\epsilon^2}{\widetilde y_i^*\widetilde x_i}\Delta y_i^*(A_k-\lambda_k^{(i)}I)\Delta x_i ,\]
yielding a rough estimate
\begin{equation}\label{eq:err2SRQ}
|\widetilde\blambda_{\mathrm{RQ2}}(\widetilde x_i,\widetilde y_i,\widetilde A_k)|\le \|E_k\|_2 \|y_i\|_2 \epsilon +  |\Delta y_i^*(A_k-\lambda_k^{(i)}I)\Delta x_i|\epsilon^2.
\end{equation}
Comparing~\eqref{eq:err2SRQ} to \eqref{eq:err_1sRQ}, we see that the error of the two-sided Rayleigh quotient depends on the condition number of a joint eigenvalue. On the other hand, the second term in \eqref{eq:err2SRQ} involves $\Delta x_i$ and $\Delta y_i$. For sufficiently small $\epsilon$, if we know that $\|\Delta x_i\|_2$ and $\|\Delta y_i\|_2$ are independent of $\epsilon$, this term should be smaller
than the first term in \eqref{eq:err_1sRQ}. We will present a more detailed comparison between $\widetilde\blambda_{\mathrm{RQ2}}^{(i)}$ and $\widetilde\blambda_{\mathrm{RQ1}}^{(i)}$ in Section~\ref{sec:main}.
\smallskip

If ${\cal A}$ is not simultaneously diagonalizable, it follows from the discussion at the end of Section~\ref{sec:commfam} that  $A(\mu)$ is not diagonalizable.
In such case there is no invertible matrix $X$ of right eigenvectors for $A(\mu)$.
Under perturbations, such $X$ almost surely exists but can be expected to be extremely ill-conditioned, so there is little hope of getting
highly accurate approximations for all eigenvalues of ${\cal A}$ from~\eqref{eq:2SRQ}. However, we will 
show in Example~\ref{ex:ex5} that even in such an extreme situation one still 
obtains meaningful approximations. We also note that when \eqref{eq:2SRQ} cannot be computed numerically, it is more reliable to resort to \eqref{eq:1sRQ}, which is well defined even for a singular $X$.

\section{Rayleigh quotient approximation error}\label{sec:main}

In this section, we analyze the quality of the Rayleigh quotient approximations produced by Algorithm~\ref{alg:RJEA} applied to a nearly commuting family $\mathcal{\widetilde A}$. As in Algorithm~\ref{alg:RJEA}, we will assume that the eigenvector matrix $\widetilde X$ of $\widetilde A(\mu)$ is normalized such that its columns have norm $1$. This scaling gives a quasi-optimal condition number for $X$; see~\cite[Theorem 2]{Demmel83}. In the (unlikely) case that $\widetilde A(\mu)$ has a multiple eigenvalue, we will assume that this eigenvalue is semisimple and $X$ contains an orthonormal basis of the corresponding eigenspace.

Our analysis consists of two parts. In Section~\ref{subsec:structural_bound} we will derive deterministic bounds that depend on the particular choice of $\mu$ and make no assumption on the perturbation. In Sections~\ref{subsec:prob_bound} and~\ref{subsec:gaussian_cond} we will derive probabilistic bounds assuming that $\mu$ is a random vector from the uniform distribution on the unit sphere in $\mathbb C^d$ and considering a random model for the perturbations.

\subsection{Structural bounds}\label{subsec:structural_bound}

Lemma~\ref{lem:ss_general} below provides error bounds for the one- and two-sided Rayleigh quotients returned by Algorithm~\ref{alg:RJEA}. We recall, see Lemma~\ref{lem:semisimple_partition} and the discussion in Section~\ref{sec:commfam}, that we can choose the right/left eigenbases $X_1,Y_1 \in \mathbb{C}^{n \times p}$ associated with a semisimple joint eigenvalue $\blambda$
such that there are invertible matrices $X= [X_1\ X_2] \in \mathbb{C}^{n \times n}$, $Y= [Y_1\ Y_2] \in  \mathbb{C}^{n \times n}$ satisfying
\begin{equation} \label{eq:choicexy}
X_1^* X_1 = I_p, \qquad Y^*X = I_n, \qquad 
     Y^*A_kX = \begin{bmatrix}
       \lambda_k I_p & 0 \\
        0 & A^{(k)}_{22}
    \end{bmatrix},\quad k=1,\ldots,d,
\end{equation}
for some $A_{22}^{(k)}\in\mathbb{C}^{(n-p)\times (n-p)}$.
\begin{lemma}\label{lem:ss_general}
For a commuting family $\mathcal{A} = \{A_k \in \mathbb{C}^{n \times n}\}_{k=1}^{d}$,
let $\blambda=(\lambda_1,\ldots,\lambda_d)$ be a semisimple joint eigenvalue 
of multiplicity $p$, and let $X= [X_1\ X_2] \in \mathbb{C}^{n \times n}$, $Y= [Y_1\ Y_2] \in  \mathbb{C}^{n \times n}$ be such that~\eqref{eq:choicexy} is satisfied.
Suppose that $\mu \in \mathbb{C}^d$, $\|\mu\|_2 = 1$, is such that $\lambda(\mu)$ is not an eigenvalue of $A_{22}(\mu)$. Consider a perturbed family 
    $\mathcal{\tilde A} = \{A_k  + \epsilon E_k\}_{k=1}^{d}$ with $\|E_1\|_F^2 + \cdots + \|E_d\|_F^2 = 1$ and
    $\epsilon>0$, such that the $p$ eigenvalues of $\widetilde A(\mu)$ closest to $\lambda(\mu)$ remain semisimple.
    Suppose that $\widetilde X_1=[\widetilde x_1\ \cdots\ \widetilde x_p]$ and 
    $\widetilde Y_1=[\widetilde y_1\ \cdots\ \widetilde y_p]$ contain the corresponding right/left eigenvectors
    of $\widetilde A(\mu)$, normalized such that $\widetilde Y_1^*\widetilde X_1=I_p$ and the columns of $\widetilde X_1$ have norm $1$.
    Then the inequalities
    \begin{align}
    |\widetilde x_i^* \widetilde A_k \widetilde x_i - \lambda_k| 
    & \le \left(\|E_k\|_2+ \|X_2\|_2\|Y_2\|_2\|D_k(\mu)\|_2\right)\epsilon+ 
    \mathcal{O}(\epsilon^2), \label{eqlem:cluster1s}\\
    |\widetilde y_i^* \widetilde A_k \widetilde x_i -\lambda_k| &\le \|\widetilde y_i\|_2 \|E_k\|_2\,\epsilon  + 
    \mathcal{O}(\epsilon^2) \label{eqlem:cluster}
   \end{align}
   hold for every $i=1,\ldots,p$ and $k=1,\ldots,d$, 
    where 
    \begin{equation}
     D_k(\mu) =\big(A_{22}^{(k)}-\lambda_k I_{n-p}\big)\big(
        A_{22}(\mu)-\lambda(\mu) I_{n-p}\big)^{-1}. \label{eqlem:Dmucluster}  
    \end{equation}
\end{lemma}

\begin{proof}
By definition, the columns of $\widetilde X_1$ form a basis for an invariant subspace 
of $\widetilde A(\mu)$. The perturbation result of Lemma~\ref{lemma:first_order_pert_cluster} provides another basis $\widehat X_1$
of the same subspace such that
$\widehat X_1 = X_1   +\epsilon \Delta X_1 + \mathcal{O}(\epsilon^2)$ with
\begin{equation}
\Delta X_1 = - X_2 (A_{22}(\mu)-\lambda(\mu) I_{n-p})^{-1} Y_2^*E(\mu) X_1. \label{thm:pert_cluster_x}
\end{equation}
Because $\widehat X_1$ and $\widetilde X_1$ span the same subspace, there is an invertible matrix $C_1$ such that
$\widetilde X_1 = \widehat X_1 C_1$. 
Moreover, there is a basis $\widehat Y_1$ of the space spanned by the columns of $\widetilde Y_1$ such that $\widehat Y_1^*=Y_1^*+\epsilon\Delta Y_1^* + \mathcal{O}( \epsilon^2)$ with
        \[
        \Delta Y_1^* = - Y_1^* E(\mu) X_2 (A_{22}(\mu)-\lambda(\mu) I_{n-p})^{-1} Y_2^*
    \]
and $\widehat Y_1^* \widehat X_1 = I_p + \mathcal{O}(\epsilon^2)$. Note that the latter relation implies $\widetilde Y_1 = \widehat Y_1 (C_1^{-*} + \mathcal{O}(\epsilon^2))$.

Letting $e_i$ denote the $i$th unit vector of length $p$ and using that $\|\widetilde x_i\|_2=1$, we obtain 
for the one-sided Rayleigh quotient that
\begin{align*}
        \widetilde x_i^*\widetilde A_k \widetilde x_i - \lambda_k & = 
        \epsilon \widetilde x_i^*E_k\widetilde x_i + \widetilde x_i^*\big(A_k-\lambda_kI\big)\widetilde x_i \\
        & = \epsilon \widetilde x_i^*E_k\widetilde x_i + \widetilde x_i^* \big(A_k-\lambda_kI\big)(X_1+\epsilon \Delta X_1)C_1e_i + \mathcal{O}(\epsilon^2) \\
        & = \epsilon \widetilde x_i^*E_k\widetilde x_i + \epsilon \widetilde x_i^* \big(A_k-\lambda_kI\big) \Delta X_1\, C_1e_i + \mathcal{O}(\epsilon^2) \\
        & = \epsilon \widetilde x_i^*E_k\widetilde x_i - \epsilon \widetilde x_i^* \big(A_k-\lambda_kI\big)  X_2 (A_{22}(\mu)-\lambda(\mu) I_{n-p})^{-1} Y_2^*E(\mu)  \widetilde x_i  + \mathcal{O}(\epsilon^2) \\
        & = \epsilon \widetilde x_i^*E_k\widetilde x_i -\epsilon \widetilde x_i^*X_2D_k(\mu)Y_2^*E(\mu)\widetilde x_i + \mathcal{O}(\epsilon^2),
    \end{align*}
    where we used $\big(A_k-\lambda_kI\big)X_1=0$ for the third 
    and  
$\big(A_k-\lambda_kI\big)X_2=X_2\big(A_{22}^{(k)}-\lambda_kI_{n-p}\big)$ for the last equality. It follows from $Y_2^* X_2 = I_{n-p}$ and
$\|E(\mu)\|_2\le 1$  that
\begin{equation*}|\widetilde x_i^* \widetilde A_k \widetilde x_i - \lambda_k^{(i)}| 
\le \left(\|E_k\|_2+ \|X_2\|_2\|Y_2\|_2\|D_k(\mu)\|_2\right)\epsilon+  \mathcal{O}(\epsilon^2),\label{eqlem:cldelta}
\end{equation*} which proves
 \eqref{eqlem:cluster1s}.

For the two-sided Rayleigh quotient, we get \eqref{eqlem:cluster} from
\begin{align*}
        \widetilde y_i^*\widetilde A_k \widetilde x_i - \lambda_k^{(i)} &= 
        \epsilon \widetilde y_i^*E_k\widetilde x_i + \widetilde y_i^*\big(A_k-\lambda_kI\big)\widetilde x_i\\
&= \epsilon \widetilde y_i^*E_k\widetilde x_i +
e_i^*C_1^{-*}(Y_1^*+\epsilon \Delta Y_1^*)\big(A_k-\lambda_kI\big)(X_1+\epsilon \Delta X_1)C_1e_i + \mathcal{O}(\epsilon^2) \\
& = \epsilon \widetilde y_i^*E_k\widetilde x_i +  \mathcal{O}(\epsilon^2)
 \end{align*}
 where we used $\widetilde y_i^*\widetilde x_i=1$, 
$Y_1^*(A_k -\lambda_kI) = 0$, and $ (A_k -\lambda_kI)X_1 = 0$.
\end{proof}

The following theorem turns the error bounds of Lemma~\ref{lem:ss_general} for individual eigenvalue components
into error bounds for the whole vector containing a joint eigenvalue.
For one-sided Rayleigh quotients, we additionally assume diagonalizability in order to obtain a convenient expression for $\|D_k(\mu)\|_2$.

\begin{theorem}\label{thm:semisimple}
    Under the notation and assumptions of Lemma \ref{lem:ss_general}, the following statements hold:
\begin{enumerate}
    \item[1)] The two-sided Rayleigh quotients
    $\widetilde{\blambda}_{\mathrm{RQ2}}^{(i)}:=
        (\widetilde y_i^* \widetilde A_1 \widetilde x_i, \ldots, \widetilde y_i^* \widetilde A_d \widetilde x_i)
    \in \mathbb{C}^d$ satisfy
    \begin{equation}\label{eqdet:Asemisimple}
        \|\widetilde{\blambda}_{\mathrm{RQ2}}^{(i)} - \blambda\|_2 \le \| \widetilde y_i \|_2 \epsilon 
         + \mathcal{O}(\epsilon^2), \qquad i=1,\ldots,p.
    \end{equation}
\item[2)] 
Additionally, assume that $\mathcal{A}$ is diagonalizable and let
    $X= [X_1\ X_2]$, $Y= [Y_1\ Y_2]$ be such that $Y^*A_kX$ is diagonal for $k = 1,\ldots,d$.
    Then the one-sided Rayleigh quotients
        $\widetilde{\blambda}_{\mathrm{RQ1}}^{(i)}:=
        (\widetilde x_i^* \widetilde A_1 \widetilde x_i, \ldots, \widetilde x_i^* \widetilde A_d \widetilde x_i)
    \in \mathbb{C}^d$
    satisfy
    \begin{equation}\label{eqdet:Asimple1s}
        \|\widetilde{\blambda}_{\mathrm{RQ1}}^{(i)} - \blambda\|_2 \le 
        \Big(1 + \frac{\sqrt{d}\|X_2\|_2\|Y_2\|_2}{d(\mu)}\Big)\epsilon 
        + \mathcal{O}(\epsilon^2), \qquad i=1,\ldots,p,
    \end{equation}
    where
   \[ 
    d(\mu):=\min_{\blambda^{(j)} \neq \blambda}\left(\frac{|\lambda_j(\mu) - \lambda(\mu)|}{\|\blambda^{(j)}- \blambda\|_2}\right),
    \]
    and $\blambda^{(1)}, \ldots, \blambda^{(n)}$ denote the joint eigenvalues of $\mathcal A$.
\end{enumerate}    
\end{theorem}
\begin{proof}
    The first result~\eqref{eqdet:Asemisimple} is obtained from~\eqref{eqlem:cluster} by summing the squared errors  and using $\|E_1\|_F^2 + \cdots + \|E_d\|_F^2 = 1$.
    
    The additional assumption for the second result implies that the matrix $D_k(\mu)$~\eqref{eqlem:Dmucluster} is diagonal and, hence,
\[\|D_k(\mu)\|_2=\max_{\blambda^{(i)} \neq  \blambda}\left(\frac{|\lambda_k^{(i)}-\lambda_k|}{|\lambda_i(\mu)-\lambda(\mu)|}\right)
\le \max_{\blambda^{(i)} \neq  \blambda}\left(\frac{\|\blambda^{(i)}- \blambda\|_2}{|\lambda_i(\mu)- \lambda(\mu)|}\right)=\frac{1}{d(\mu)}.\]
Now,~\eqref{eqdet:Asimple1s} is again obtained from~\eqref{eqlem:cluster1s} by summing the squared errors. 
\end{proof}

The bounds of Theorem~\ref{thm:semisimple} contain three factors that require explanation and additional considerations:
\begin{enumerate}
 \item The factor $\|\widetilde y_i\|_2$ in~\eqref{eqdet:Asemisimple} is the condition number of the $i$th eigenvalue of $\widetilde A(\mu)$. For $p = 1$, this quantity converges to the condition number of $\blambda$ as $\epsilon \to 0$; see Definition~\ref{def:eigenvalue_condition_joint_semisimple} in Section~\ref{sec:eigenvalue_condition_number}.
For $p>1$, the interpretation of this quantity is much more subtle; see Section~\ref{subsec:gaussian_cond} below.

\item The factor $\|X_2\|_2\|Y_2\|_2$ in~\eqref{eqdet:Asimple1s} quantifies how well $\mathcal A$ can be diagonalized because it is bounded by the condition number of the matrix $X$ diagonalizing $\mathcal{A}$. In the context of the bound~\eqref{eqlem:cluster1s} from Lemma~\ref{lem:ss_general}, $X$ only needs to block diagonalize
$\mathcal A$ and, in turn, $\|X_2\|_2\|Y_2\|_2$ can be chosen to be significantly smaller. In particular, for $p = 1$ we may assume, without loss of generality, that $x_1 = e_1$ 
and $y_1 = \big[ \genfrac{}{}{0pt}{}{1}{v} \big]$ for some $v \in \mathbb C^{n-1}$. 
Then we can choose $X_2 = \big[ \genfrac{}{}{0pt}{}{-v^*}{I_{n-1}}\big]$, for which we have
$\|X_2\|_2\|Y_2\|_2 = \|y_1\|_2$. In other words, $\|X_2\|_2\|Y_2\|_2$ also becomes the condition number of $\blambda$.
\item For unfortunate choices of $\mu$, the quantity $d(\mu)$ in~\eqref{eqdet:Asimple1s} becomes arbitrarily small. In order to keep $1/d(\mu)$ bounded, it is essential to choose $\mu$ randomly; we will analyze this choice in the next section.
\end{enumerate}

\subsection{A probabilistic bound for one-sided Rayleigh quotients}\label{subsec:prob_bound}

In order to investigate the quantity $d(\mu)$ in the error bound~\eqref{eqdet:Asimple1s} for one-sided Rayleigh quotients, we now assume that
$\mu \sim \mathrm{Unif}(\mathbb{S}_{\mathbb{C}}^{d-1})$, i.e., $\mu$ is from the uniform distribution on the unit sphere in $\mathbb C^d$
as chosen in Algorithm~\ref{alg:RJEA}.
For this purpose, we will use anti-concentration results.
\begin{lemma}[{\cite[Lemma 5.2]{banks2022global}}]\label{lemma:anti_concentration}
    Let $\mu \sim \mathrm{Unif}(\mathbb{S}_{\mathbb{C}}^{d-1})$ and $v \in \mathbb{S}_{\mathbb{C}}^{d-1}$. Then
    $\Prob\left(|\mu^*v| \leq t/ \sqrt{d-1} \right) \leq t^2$
    holds for all $t \in [0,1]$.
\end{lemma}
 
\begin{lemma}\label{lem:probabilistic_bounds}
For $\mu \sim \mathrm{Unif}(\mathbb{S}_{\mathbb{C}}^{d-1})$, the quantity $d(\mu)$ defined in Theorem~\ref{thm:semisimple}
satisfies
    \begin{equation*}\label{lemeq:d_mu_prob_bound}
        \Prob\big(d(\mu) \ge C\big)  \ge 1-  (n-p) (d-1)C^2.
    \end{equation*}
\end{lemma}
\begin{proof}
From the definition of $d(\mu)$, we obtain that
   \begin{align*}
       \Prob\big(d(\mu) \ge C\big) &= \Prob\Big(\min_{\blambda^{(j)} \not= \blambda} \frac{|\lambda_j(\mu) -  \lambda(\mu)|}{\|\blambda^{(j)}-  \blambda\|_2} \ge C\Big) \nonumber\\
       &= \Prob\Big(\min_{\blambda^{(j)} \not= \blambda} \frac{|\mu^*(\blambda^{(j)} -  \blambda)|}{\|\blambda^{(j)} -  \blambda\|_2} \ge C\Big) \ge 1 - (n-p) (d-1)C^2.
   \end{align*} 
   The last inequality uses Lemma~\ref{lemma:anti_concentration} combined with the union bound, noting that at most $n-p$ eigenvalues are different
   from $\blambda$.
\end{proof}

We now use Lemma~\ref{lem:probabilistic_bounds} to turn the result of Theorem \ref{thm:semisimple} 2) into a probabilistic bound.
\begin{theorem}\label{cor:prob_bound_semisimple}
    Under the notation and assumptions of Theorem~\ref{thm:semisimple} 2), let 
    $\mu\sim \mathrm{Unif}(\mathbb{S}_{\mathbb{C}}^{d-1})$.
    Then for every $R > 0$ it holds for $i=1,\ldots,p$ that
    \begin{equation}\label{eq:boundR_semisimple}
        \Prob\big( \big\|\widetilde{\blambda}_{\mathrm{RQ1}}^{(i)} - \blambda \big\|_2 \le 
        (1 + \|X_2\|_2\|Y_2\|_2R)\epsilon + \mathcal{O}(\epsilon^2) \big) \ge 1 - \frac{(n-p)(d-1)d}{R^2}.
    \end{equation}
\end{theorem}
\begin{proof}
Setting $\widetilde R=\|X_2\|_2\|Y_2\|_2 R$, the inequality~\eqref{eqdet:Asimple1s} from Theorem \ref{thm:semisimple} yields
    \begin{align*}
        \Prob\big(\big\| \widetilde{\blambda}_{\mathrm{RQ1}}^{(i)} &- \blambda \big\|_2 \le 
        (1 + \widetilde R)\epsilon + \mathcal{O}(\epsilon^2) \big)\\ 
       &\geq  \Prob\left(\frac{\sqrt{d}\|X_2\|_2\|Y_2\|_2} {d(\mu)} \leq \widetilde R \right) 
         = \Prob\left(d(\mu) \geq \frac{\sqrt{d}\|X_2\|_2\|Y_2\|_2}{\widetilde R} \right)\\
        &\geq 1 - (n-p) (d-1) \frac{d\|X_2\|^2_2\|Y_2\|^2_2}{\widetilde R^2} = 1 - \frac{(n-p)(d-1)d}{R^2},&
    \end{align*}
    where the last inequality uses Lemma~\ref{lem:probabilistic_bounds}. 
\end{proof}

The inequality~\eqref{eq:boundR_semisimple} mixes a probabilistic with an asymptotic bound, and its interpretation requires some care. We first draw $\mu\sim \mathrm{Unif}(\mathbb{S}_{\mathbb{C}}^{d-1})$. Then there is a constant $C_\mu$, independent of $\epsilon$ but possibly depending on $\mu$, such that
the approximation error of the one-sided Rayleigh quotient is bounded by $(1 +\|X_2\|_2\|Y_2\|_2R)\epsilon + C_\mu \epsilon^2$ for every $\epsilon > 0$ with a failure probability that decreases proportionally to $1/R^2$ when the extra factor $R$ increases.

From the discussion after Theorem~\ref{cor:prob_bound_semisimple}, we know that, for $p = 1$, 
the approximation error of the two-sided Rayleigh quotient is bounded by $\|y_1\|_2 \epsilon + \widetilde C_\mu \epsilon^2$. As 
$\|y_1\|_2 \le\|X_2\|_2\|Y_2\|_2$, our analysis indicates that the two-sided Rayleigh quotient is more accurate for simple, well isolated eigenvalues. Because of second-order terms, this analysis does not extend to clusters of very close eigenvalues. Nevertheless, also for such cases numerical experiments reveal that the one-sided approach is not more reliable than the two-sided approach; see Example~\ref{ex:ex3} below for more details.

\subsection{On two-sided Rayleigh quotients for semisimple eigenvalues}\label{subsec:gaussian_cond}

As mentioned after Theorem~\ref{thm:semisimple}, the factor $\|\widetilde y_i\|_2$, featured in the error bound~\eqref{eqdet:Asemisimple} for two-sided Rayleigh quotients, approaches the eigenvalue condition number in the case of a simple eigenvalue. This quantity becomes much harder to interpret for a semisimple eigenvalue $\blambda$ of multiplicity $p>1$, because the perturbations may turn $\blambda$ into $p$ arbitrarily ill-conditioned perturbed eigenvalues . To better understand this phenomenon, we first note that
$\widetilde y_i=\widetilde Y_1 e_i = \widehat Y_1 C_1^{-*}e_i = Y_1C_1^{-*}e_i + \mathcal{O}(\epsilon)$
implies
\begin{equation} \label{eq:boundtildey}
 \|\widetilde y_i\|_2 \lesssim \|Y_1\|_2 \cdot \|C_1^{-*}e_i\|_2 \le \|Y_1\|_2 \cdot \|C_1^{-*}\|_2.
\end{equation}
We recall that $C_1$ is the matrix defined in the proof of Theorem~\ref{thm:semisimple},
linking the perturbed invariant subspace $\widehat X_1$, defined in~\eqref{thm:pert_cluster_x}, to the computed basis $\widetilde X_1$ of $\widetilde A(\mu)$.
Because the columns of $X_1$ are orthonormal, it follows that every column of $C_1$ has norm $1+{\cal O}(\epsilon)$. Moreover, by definition, the matrix $\widetilde Y_1^* \widetilde A(\mu) \widetilde X_1$ is diagonal and satisfies
\begin{align*}
 \widetilde Y_1^* \widetilde A(\mu) \widetilde X_1 - \lambda(\mu) I_p &= 
 \widetilde Y_1^* (\widetilde A(\mu) -\lambda(\mu) I) \widetilde X_1\\ &=
 C_1^{-1} \widehat Y_1^* (\widetilde A(\mu) -\lambda(\mu) I) \widehat X_1 C_1 
 = \epsilon C_1^{-1} Y_1^*E(\mu)X_1 C_1 + \mathcal O(\epsilon^2).
\end{align*}
Hence, in first order, $C_1$ is a matrix that diagonalizes $Y_1^*E(\mu)X_1$ and its columns have norm $1$. 

Generically, the $p$ eigenvalues of $Y_1^*E(\mu)X_1$ are mutually distinct eigenvalues and therefore this matrix is almost surely diagonalizable. Still, the matrix $C_1$ can become arbitrarily ill-conditioned. To be able to say more about $\|C_1^{-*}\|_2$, we need to impose some assumptions on the perturbations. In the following, we will consider random perturbations. Specifically, we will consider the case that $Y_1^*E(\mu)X_1$ is a (scaled) complex Gaussian random matrix.

In the following, we will consider a $p\times p$ complex Gaussian random matrix $G_p = (g_{ij})$, scaled such that the real and imaginary parts of its entries $g_{ij}$ are i.i.d. $\mathcal{N}(0, 1/2p)$. The scaling by $2p$ ensures that the norm does not grow with $p$.
\begin{lemma}[{\cite[Lemma 2.2]{gaussaincondition}}] \label{lemma:gaussian_operator_norm}
    The $p \times p$ complex Gaussian matrix $G_p$ defined above satisfies
    \[\Prob\big( \| G_p \|_2 >  (2\sqrt{2} + t) \big) \leq 2 \exp(-pt^2) \]
    for every $t > 0 $ and $R > 0$, 
\end{lemma}

The following result is a corollary of Theorem 1.5 in~\cite{gaussaincondition} (obtained by setting $A = 0$ and $\delta = 1$ in the statement of that theorem).
\begin{lemma} \label{lemma:eigenvalue_cond_bound}
    Let $\lambda_1, \ldots, \lambda_p \in \mathbb{C}$ denote the eigenvalues of a complex Gaussian matrix $G_p$. Then, for every measurable open set $B \subset \mathbb{C}$, the condition number of $\lambda_i$ satisfies
    \[\expct \int_{\lambda_i \in B} \kappa(\lambda_i)^2 \leq \frac{p^2\mathrm{vol}(B)}{\pi}.\]
\end{lemma}

Combining the above two results allows us to control the eigenvector matrix condition number of  $G_p$.

\begin{theorem}\label{thm:gaussian_eigenvector_cond}
Let $X$ be a $p\times p$  matrix containing the eigenvectors of a $p\times p$ complex Gaussian matrix $G_p$ defined above, scaled such that the columns of $X$ have unit norm. Then
        \[ \Prob\left( \kappa_2(X) \geq R \right) \leq 2 \exp(-p t^2) +  \frac{p^3(2\sqrt{2}+t)^2}{R^2} \]
        holds for every $t>0$ and $R>0$. 
\end{theorem}
\begin{proof}
Letting $B$ be the ball of radius $\rho$ centered at the origin, we can conclude from Lemma~\ref{lemma:eigenvalue_cond_bound} that
     \begin{equation}\label{eq:condition_expectation}
     \expct \int_{\lambda_i \in B} \kappa(\lambda_i)^2 \leq p^2 \rho^2.
     \end{equation}
By \cite[Lemma 3.1]{gaussaincondition},
$\kappa_2(X)^2 \leq  p( \kappa(\lambda_1)^2 + \cdots + \kappa(\lambda_p)^2)$.
Applying Markov's inequality to $\eqref{eq:condition_expectation}$ gives 
    \[\Prob\left( \kappa_2(X)^2 / p \geq s^2 \cap  \|G_p\|_2 \le \rho \right) \leq p^2 \rho^2 / s^2,\]
    or, equivalently,
    \begin{equation*}\label{eq:markov_eig_cond}
        \Prob\left( \kappa_2(X)  \geq R  \cap  \|G_p\|_2 \le \rho \right) \leq p^3 \rho^2 / R^2,
    \end{equation*}
where we used that $\|G_p\|_2 \le \rho$ implies $|\lambda_i| \leq \rho$. The proof is completed by using
\[ \Prob\left( \kappa_2(X)  \geq R \right) \le \Prob\left( \kappa_2(X)  \geq R  \cap  \|G_p\|_2 \le \rho \right) + \Prob\left(\|G_p\|_2 \ge \rho \right) \]
and applying Lemma~\ref{lemma:gaussian_operator_norm} with $\rho = 2\sqrt{2} + t$.
\end{proof}

Theorem~\ref{thm:gaussian_eigenvector_cond} shows that the eigenvector matrix of a $p\times p$ complex Gaussian matrix has a condition number $\mathcal{O}(p^{3/2})$ with high probability, implying that it is extremely unlikely to encounter very large condition numbers. Combined with~\eqref{eq:boundtildey}, this \emph{indicates} that the condition numbers $\|\widetilde y_i\|_2$ of the perturbed eigenvalues are likely not much larger than the condition number $\|Y_1\|_2$ of the unperturbed semisimple eigenvalue. This expectation is matched by the numerical experiments reported in Section~\ref{sec:synthetic}.

\section{Numerical experiments for synthetic data}\label{sec:synthetic}

In this section, we report numerical experiments using synthetic data that highlight various theoretical aspects of this work. For this purpose, we generate pairs of commuting matrices $A_1$ and $A_2$ with known joint eigenvalues by performing a similarity transformation of diagonal matrices.
The transformation is effected by a random $n\times n$ matrix $X$ with normalized columns and prescribed condition number $\kappa_2(X)=\kappa$. In this regard, 
we generate $\widetilde X=Q_1DQ_2$, where $D={\rm diag}(1,\widetilde\kappa^{1/(n-1)},\widetilde\kappa^{2/(n-1)},\ldots,\widetilde\kappa)$ and $Q_1,Q_2$ are random orthogonal $n\times n$ matrices.
The matrix $X$ is obtained from $\widetilde X$ by normalizing its columns, and $\widetilde \kappa$ is adjusted such that $\kappa_2(X)=\kappa$.
We generate nearly commuting matrices $\widetilde A_i=A_i+\epsilon \sqrt{2}/2 E_i$ for $i =1,2$ by adding Gaussian random matrices  $E_i$, scaled  such that $\|E_i\|_F=1$. 
To compute eigenvalue approximations, we apply Algorithm~\ref{alg:RJEA} repeatedly with $N=10^4$ realizations of the random vector $\mu$.
All numerical experiments were carried out in Matlab 2023a. The code and data for numerical examples in this paper are available 
at \url{https://github.com/borplestenjak/RandomJointEig}.

\begin{example}\rm\label{ex:ex1} 
We take $A_i=XD_iX^{-1}$ with
$D_1={\rm diag}(1,1,1,2,2,2,3)$, $D_2={\rm diag}(1,2,3,1,2,3,3)$,
and $\kappa_2(X)=10^2$. Both matrices $A_1$ and $A_2$ have multiple eigenvalues, but the joint eigenvalues are all simple.
The condition number of the joint eigenvalue $\blambda^{(1)}=(1,1)$ is  $21.6$, $\|A_1\|_2=44.8$ and $\|A_2\|_2=38.6$, representing a well-conditioned problem. 
We generated nearly commuting pairs $\widetilde A_1,\widetilde A_2$ with noise levels $\epsilon=0,10^{-14},10^{-12},10^{-10}$. Even for $\epsilon=0$, the matrices are inevitably affected by roundoff error on the level  $\epsilon_0=(\|A_1\|^2_2+\|A_2\|^2_2)^{1/2}u \approx 6.6\cdot 10^{-15}$, where $u$ is the unit roundoff in double precision.

Figure~\ref{fig:Ex1_Fig} shows the distribution of the approximation errors obtained when applying  Algorithm~\ref{alg:RJEA} $N = 10^4$ times to $\widetilde A_1,\widetilde A_2$.
As $\blambda^{(1)}$ is simple, Theorem~\ref{thm:semisimple} indicates that the distribution of the errors for the two-sided Rayleigh quotient is flat because the first-order term does not depend on $\mu$.
While this is indeed true for $\epsilon=10^{-12}$ and $\epsilon=10^{-10}$, the effects of roundoff error,  committed during the computation of the Rayleigh quotients, dominate for $\epsilon=10^{-14}$ and $\epsilon=0$, leading to non-flat distributions. It can be seen from the right plot
of Figure~\ref{fig:Ex1_Fig} that these effects disappear when the computations are carried out in quadruple precision, using the Advanpix Multiprecision Computing Toolbox~\cite{advanpix}, before rounding the results back to double precision. 

\begin{figure}[htb!]
    \centering
 \includegraphics[width=6.4cm]{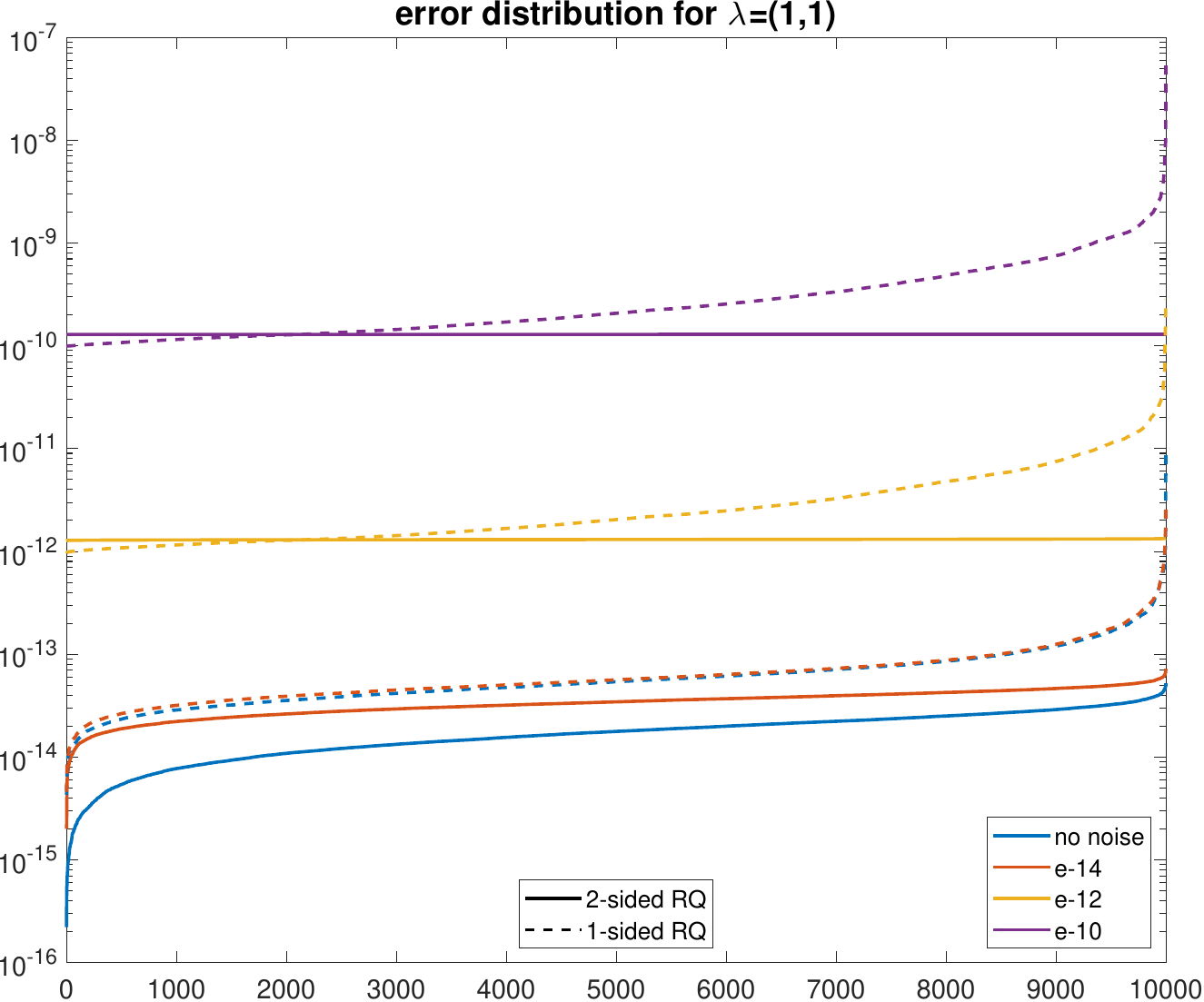}\ \
 \includegraphics[width=6.4cm]{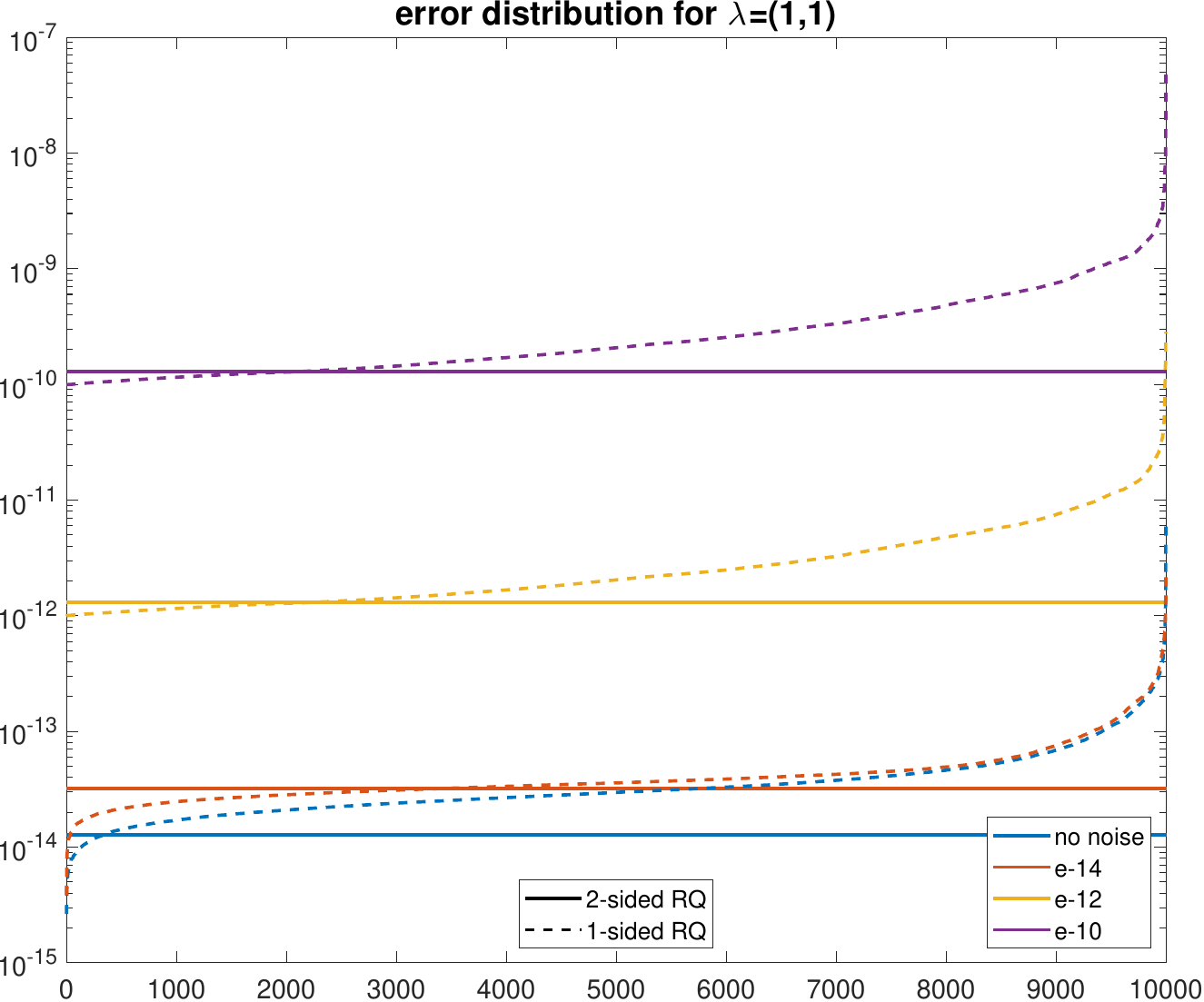}
    \caption{Distribution of absolute errors for the eigenvalue $\blambda^{(1)}$ from
    Example \ref{ex:ex1} using double precision (left) and extended precision (right) to compute the eigenvalues of $\widetilde A(\mu)$ using one-sided (dashed lines) and two-sided (solid lines) Rayleigh quotients.}
    \label{fig:Ex1_Fig}  
\end{figure}

Figure~\ref{fig:Ex1_Fig2} compares the approximation errors with the first-order terms of the bounds~\eqref{eqdet:Asemisimple} and \eqref{eqdet:Asimple1s} from Theorem \ref{thm:semisimple}. It can be observed that the bounds match the true errors within 1-2 orders of magnitude.

\begin{figure}[htb!]
    \centering
 \includegraphics[width=6.4cm]{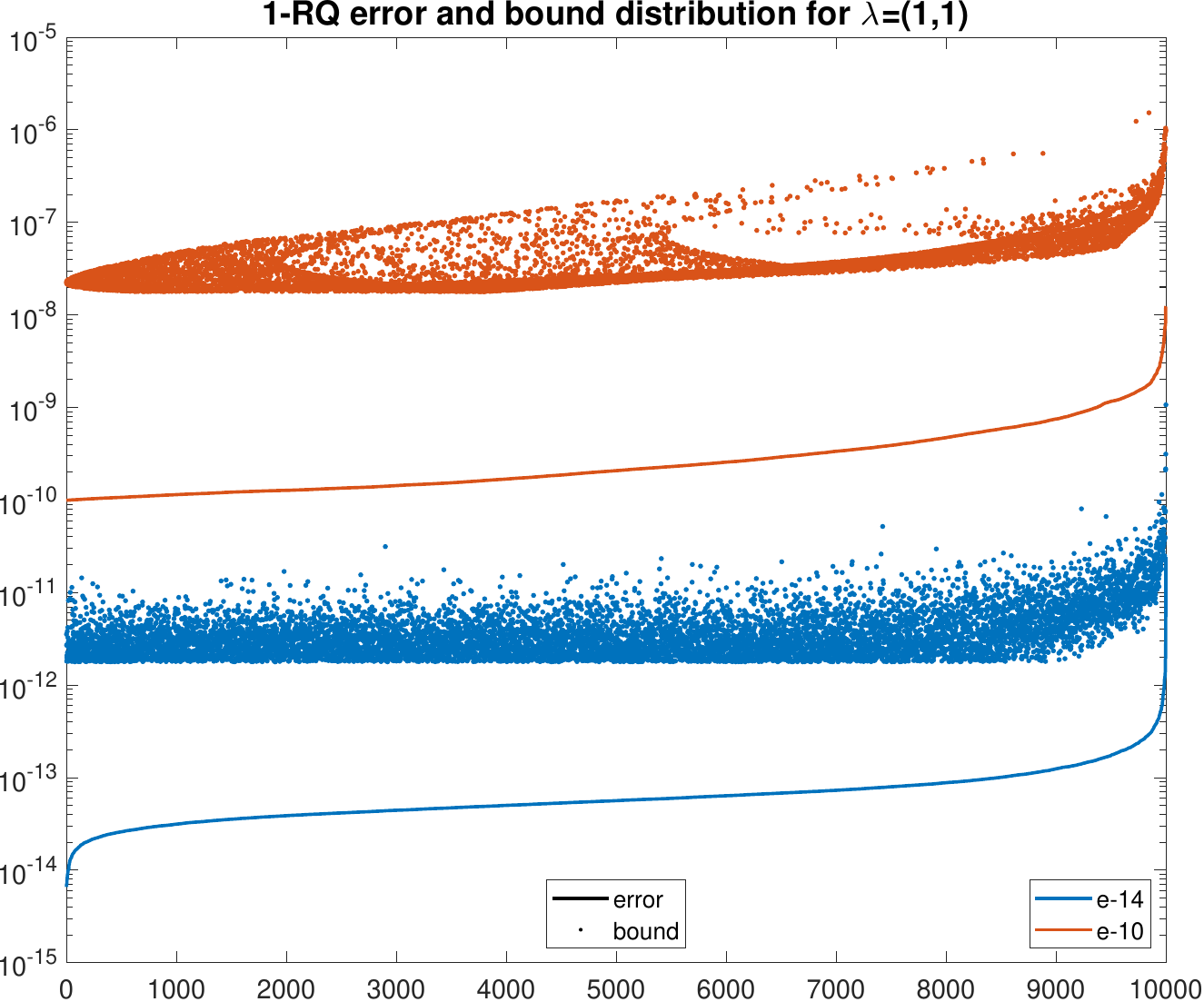}\ \
 \includegraphics[width=6.4cm]{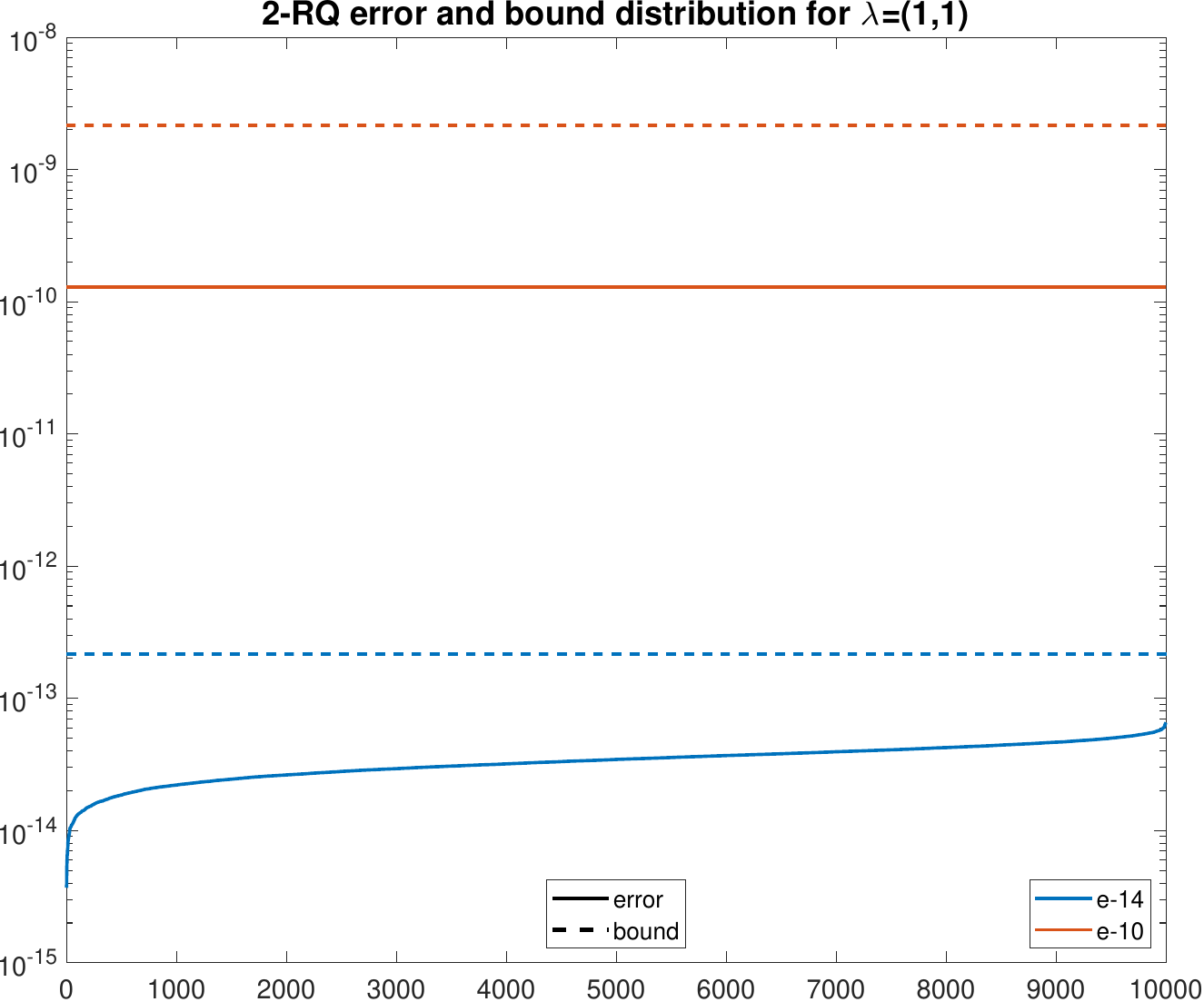}
    \caption{Absolute errors for the eigenvalue $\blambda^{(1)}$ from
    Example \ref{ex:ex1} using one-sided Rayleigh quotient vs. bound \eqref{eqdet:Asimple1s} (left) and
    two-sided Rayleigh quotient vs. bound \eqref{eqdet:Asemisimple} (right).}
    \label{fig:Ex1_Fig2}  
\end{figure}

Additional statistics are reported in Table \ref{table:Ex1}.
As predicted by Theorem~\ref{thm:semisimple},
the median error is proportional to $\epsilon$.
The last two columns show that the error of the two-sided Rayleigh quotient is often better
and nearly never more than five times worse than the error of the one-sided Rayleigh quotient.

We have also computed eigenvalue approximation errors for the other joint eigenvalues and the observations are similar to the ones reported for $\blambda^{(1)}=(1,1)$.

\begin{table}[!hbt!]
\caption{Median approximation errors, bounds~\eqref{eqdet:Asimple1s} and \eqref{eqdet:Asemisimple}, and the empirical probability that $b$ (error of two-sided Rayleigh quotient) is (significantly) smaller than $a$
(error of one-sided Rayleigh quotient) for the eigenvalue $\blambda^{(1)}$ from Example \ref{ex:ex1}.}
\label{table:Ex1}
\small
\begin{tabular}{|l|l|l|l|l|l|l|}
\hline \rule{0pt}{2.3ex}
 & \multicolumn{2}{c}{1RQ: $\blambda^{(1)} = (1,1)$} & \multicolumn{2}{|c|}{2RQ: $\blambda^{(1)} = (1,1)$}
& \multicolumn{2}{c|}{empirical prob.}\\
\hline \rule{0pt}{2.3ex}
 noise $\epsilon$ & error $(a)$ & bound \eqref{eqdet:Asimple1s} & error $(b)$ & bound \eqref{eqdet:Asemisimple}  & ${\mathbb P}(b<a)$  &  ${\mathbb P}(b<5a)$\\ 
\hline \rule{0pt}{2.3ex}
$0$        & $5.4\cdot 10^{-14}$ & $3.5\cdot 10^{-13}$ & $1.8\cdot 10^{-14}$ & $1.4\cdot 10^{-13}$ & $0.9609$ & $1.0000$\\
$10^{-14}$ & $5.6\cdot 10^{-14}$ & $3.3\cdot 10^{-12}$ & $3.4\cdot 10^{-14}$ & $2.2\cdot 10^{-13}$ & $0.8433$ & $0.9995$ \\
$10^{-12}$ & $2.0\cdot 10^{-12}$ & $2.7\cdot 10^{-10}$ & $1.3\cdot 10^{-12}$ & $2.2\cdot 10^{-11}$ & $0.7777$ & $1.0000$\\
$10^{-10}$ & $2.1\cdot 10^{-10}$ & $2.3\cdot 10^{-8}$  & $1.3\cdot 10^{-10}$ & $2.2\cdot 10^{-9}$  & $0.7865$ & $1.0000$ \\
\hline
\end{tabular}
\end{table}
\end{example}

\begin{example}\rm\label{ex:empProb}
Using the same matrices from the previous example we now test the probabilistic bounds from Section~\ref{subsec:prob_bound}. For the eigenvalue $\blambda^{(1)}=(1,1)$,
it follows from Lemma~\ref{lem:probabilistic_bounds} that 
    \begin{equation}\label{lemeq:d_mu_prob_bound_ex}
        \Prob(d(\mu) < 1/R )  < 6/ R^2.
    \end{equation}
We computed $
    d(\mu):=\min_{i=2,\ldots,7}\left(\frac{|\lambda_i(\mu) - \lambda_1(\mu)|}{\|\blambda^{(i)}-\blambda_1\|}\right)
    $ for $N=10^6$ realizations of $\mu$ and computed
the empirical probabilities for $\Prob\big(d(\mu) < 1/R\big)$ for various values of $R$.
The left plot of Figure~\ref{fig:empProb} reveals that the 
bound \eqref{lemeq:d_mu_prob_bound_ex} is quite tight.

We also computed the empirical probabilities for the result of 
Theorem~\ref{cor:prob_bound_semisimple} concerning the one-sided Rayleigh quotient approximation error.
For this purpose, we added noise of the levels
$\epsilon \in \{ 10^{-12},10^{-11},10^{-10},10^{-9} \}$ 
to create nearly commuting pairs $\widetilde A_1,\widetilde A_2$.
In line with~\eqref{eq:boundR_semisimple}, we compare the empirically measured probabilities with
    \begin{equation}\label{eq:boundR_semisimple_ex}
        \Prob\big( \|\widetilde{\blambda}_{\mathrm{RQ1}}^{(i)} - \blambda^{(1)}\|_2 > 
        (1 + R)\epsilon \big) < {12\kappa_2(X)} / {R^2}.
    \end{equation} 
The results are reported in Figure~\ref{fig:empProb} (right). This time, the empirical probabilities remain proportional to $R^{-2}$, but the bound is not as tight as on the left-hand side. As expected from Theorem~\ref{cor:prob_bound_semisimple}, the level of noise only has a negligible impact on the failure probability of the bound.

\begin{figure}[htb!]
    \centering
 \includegraphics[width=6.4cm]{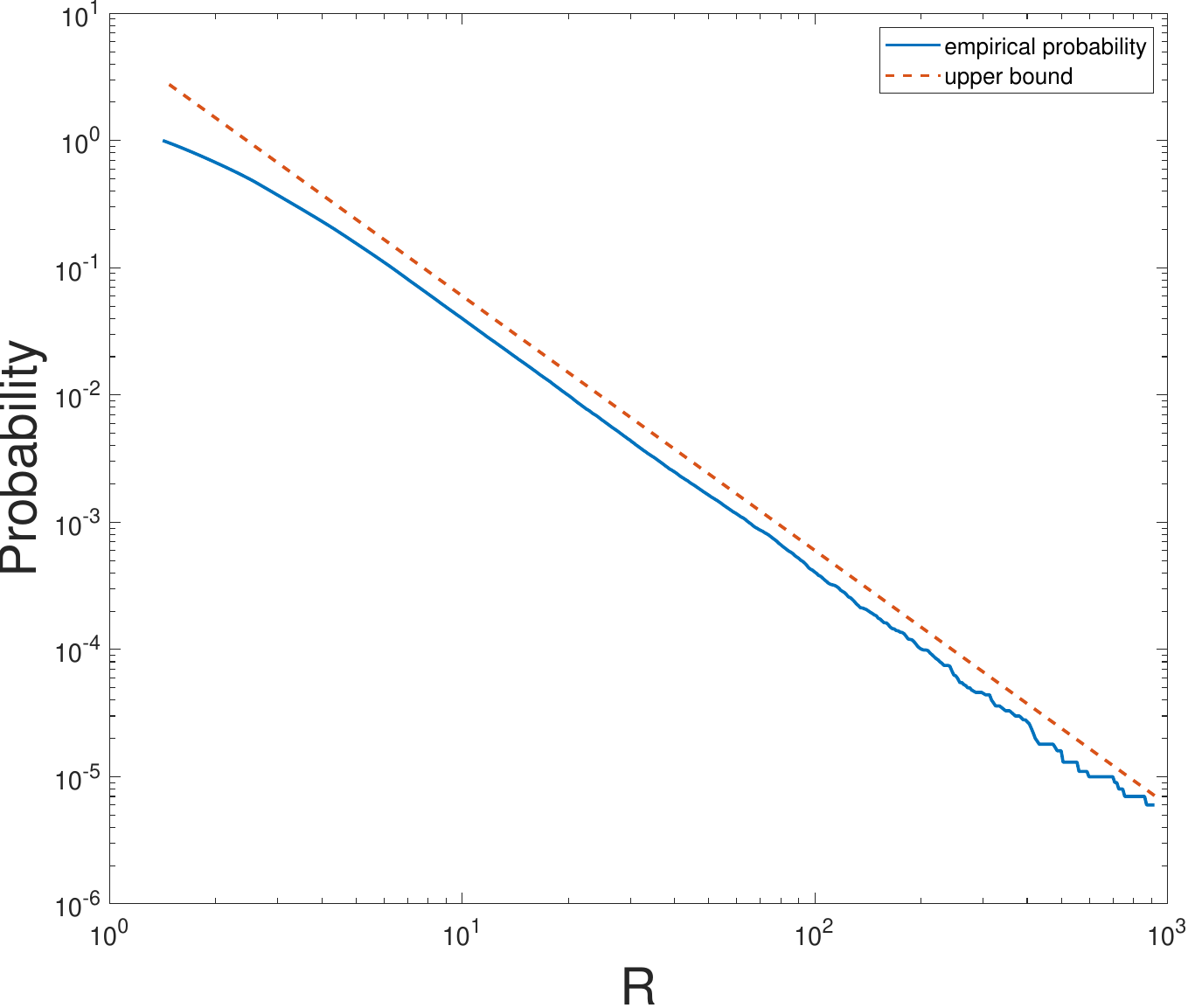}\ \
 \includegraphics[width=6.4cm]{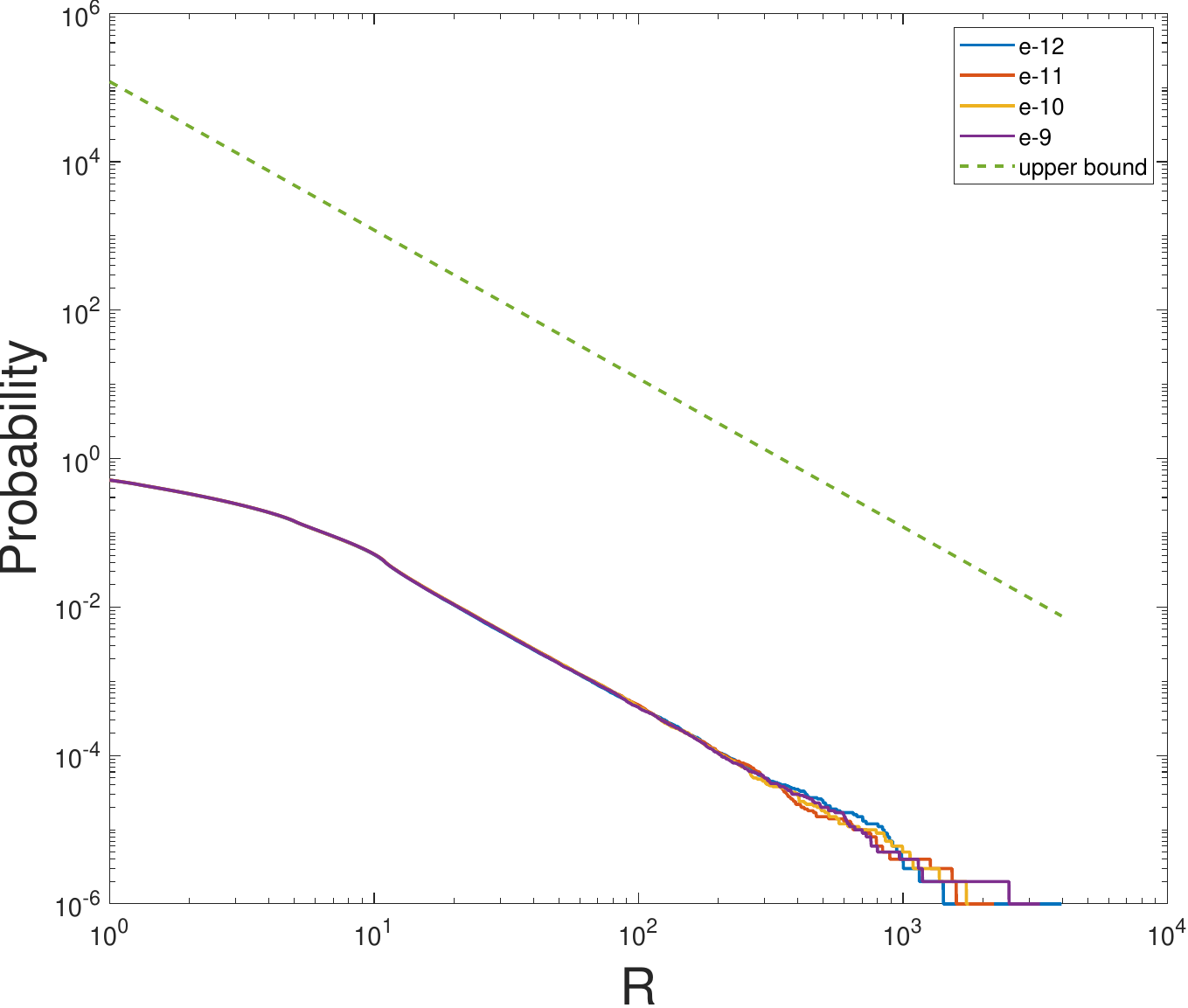}\ 
    \caption{Empirical probabilities compared to the bounds for \eqref{lemeq:d_mu_prob_bound_ex} (left) and \eqref{eq:boundR_semisimple_ex} (right) for the data from Example~\ref{ex:empProb}.}
    \label{fig:empProb}  
\end{figure}
\end{example}

\begin{example}\rm\label{ex:ex2} 
We take $A_i=XD_iX^{-1}$ with $D_i$ as in Example \ref{ex:ex1}, but now
$X=P\cdot {\rm diag}(I_2,Z)$, where $P$ is a $7\times 7$ Gaussian random matrix and $Z$ is a
$5\times 5$ random matrix with $\kappa_2(Z)=10^4$. While the joint eigenvalues remain the same, the different choice of $X$ implies that the condition numbers of $\blambda^{(1)}=(1,1)$ and $\blambda^{(2)}=(1,2)$ are much smaller than those of the other five eigenvalues.
In particular, the condition numbers of $\blambda^{(1)}$ and $\blambda^{(4)}$ are $2.1$ and $4.1\cdot 10^3$, respectively. We have 
$\|A_1\|_2=2.6\cdot 10^3$, $\|A_2\|_2=7.1\cdot 10^3$ and, hence, the noise created by roundoff error is on the level $\epsilon_0=8.4\cdot 10^{-13}$. We repeated the experiments described from Example~\ref{ex:ex1}
with noise levels $\epsilon=0,10^{-12},10^{-10},10^{-8}$; see Figure~\ref{fig:Ex2_Fig}.
For $\blambda^{(4)}$ the one-sided and two-sided Rayleigh quotients return comparable accuracy, in line with what has been observed in Example~\ref{ex:ex1}. On contrast, the two-sided Rayleigh quotient is significantly more accurate for $\blambda^{(1)}$. This reflects the fact that 
the error bound from Theorem~\ref{thm:semisimple} for the two-sided Rayleigh quotient features
$\kappa( \blambda^{(1)} ) = 2.1$, while the error bound for the one-sided Rayleigh quotient features $\|X_2\|_2 \|Y_2\|_2 \approx 10^4$. Additional statistics, demonstrating the advantages of the two-sided version, is reported in Table~\ref{table:Ex2}.

\begin{figure}[htb!]
    \centering
 \includegraphics[width=6.4cm]{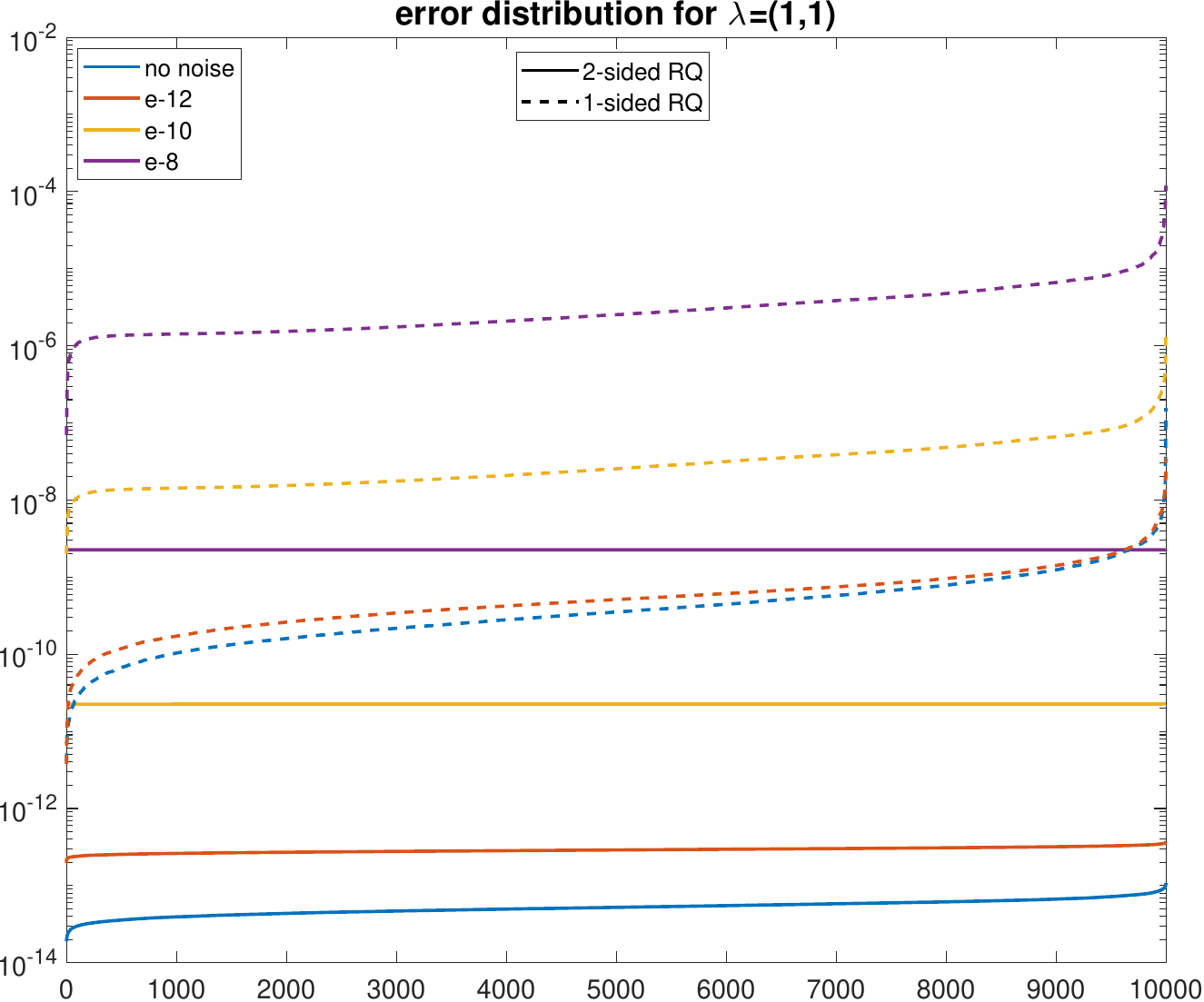}\ \
 \includegraphics[width=6.4cm]{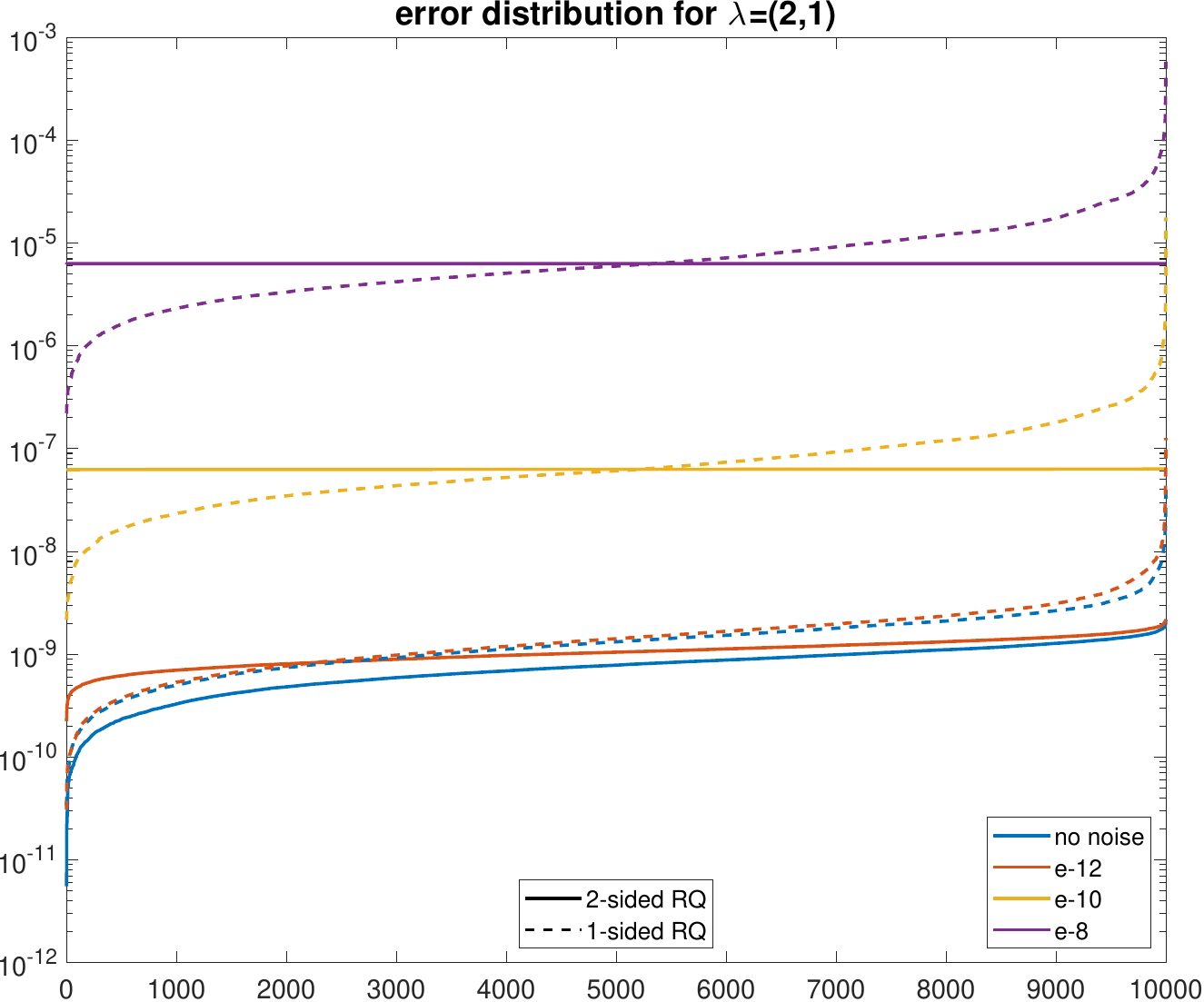}
    \caption{Distribution of absolute errors for the eigenvalues $\blambda^{(1)}=(1,1)$ (left) and $\blambda^{(4)}=(2,1)$ (right) from
    Example~\ref{ex:ex2} using one-sided (dashed lines) and two-sided (solid lines) Rayleigh quotients.
        \label{fig:Ex2_Fig}  }
\end{figure}

\begin{table}[!hbt!]
\caption{Median approximation errors, bounds~\eqref{eqdet:Asimple1s} and \eqref{eqdet:Asemisimple}, and the empirical probability that $b$ (error of two-sided Rayleigh quotient) is (significantly) smaller than $a$
(error of one-sided Rayleigh quotient) for the eigenvalues $\blambda^{(1)}$ and $\blambda^{(4)}$ from Example~\ref{ex:ex2}.}
\label{table:Ex2}
\small
\begin{tabular}{|l|l|l|l|l|l|l|}
\hline \rule{0pt}{2.3ex}
 & \multicolumn{2}{c}{1RQ: $\blambda^{(1)} = (1,1)$} & \multicolumn{2}{|c|}{2RQ: $\blambda^{(1)} = (1,1)$}
& \multicolumn{2}{c|}{empirical prob.}\\
\hline \rule{0pt}{2.3ex}
 noise $\epsilon$ & error $(a)$ & bound \eqref{eqdet:Asimple1s} & error $(b)$ & bound \eqref{eqdet:Asemisimple} & ${\mathbb P}(b<a)$  &  ${\mathbb P}(b<5a)$\\ 
\hline \rule{0pt}{2.3ex}
$0$        & $3.5\cdot 10^{-10}$ & $5.2\cdot 10^{-8}$ & $5.2\cdot 10^{-14}$ & $1.8\cdot 10^{-12}$ & $1.0000$ & $1.0000$ \\
$10^{-12}$ & $5.1\cdot 10^{-10}$ & $9.6\cdot 10^{-8}$ & $2.9\cdot 10^{-13}$ & $2.1\cdot 10^{-12}$ & $1.0000$ & $1.0000$ \\
$10^{-10}$ & $2.5\cdot 10^{-8}$  & $3.4\cdot 10^{-6}$ & $2.3\cdot 10^{-11}$ & $2.1\cdot 10^{-10}$ & $1.0000$ & $1.0000$ \\
$10^{-8}$  & $2.5\cdot 10^{-6}$  & $5.4\cdot 10^{-4}$ & $2.3\cdot 10^{-9}$  & $2.1\cdot 10^{-8}$  & $1.0000$ & $1.0000$ \\
\hline
 & \multicolumn{2}{c}{1RQ: $\blambda^{(4)} = (2,1)$} & \multicolumn{2}{|c|}{2RQ: $\blambda^{(4)} = (2,1)$}
& \multicolumn{2}{c|}{empirical prob.}\\
\hline \rule{0pt}{2.3ex}
 noise $\epsilon$   & error $(a)$ & bound \eqref{eqdet:Asimple1s} & error $(b)$ & bound \eqref{eqdet:Asemisimple} & ${\mathbb P}(b<a)$ &  ${\mathbb P}(b<5a)$ \\ 
\hline \rule{0pt}{2.3ex}
$0$        & $1.3\cdot 10^{-9}$  & $1.4\cdot 10^{-8}$ & $7.8\cdot 10^{-10}$ & $3.5\cdot 10^{-9}$ & $0.7519$ & $0.9886$\\
$10^{-12}$ & $1.4\cdot 10^{-9}$  & $1.6\cdot 10^{-8}$ & $1.0\cdot 10^{-9}$  & $4.2\cdot 10^{-9}$ & $0.6580$ & $0.9831$ \\
$10^{-10}$ & $6.1\cdot 10^{-8}$  & $1.6\cdot 10^{-6}$ & $6.3\cdot 10^{-8}$  & $4.2\cdot 10^{-7}$ & $0.4784$ & $0.9715$ \\
$10^{-8}$  & $6.0\cdot 10^{-6}$  & $6.1\cdot 10^{-4}$ & $6.3\cdot 10^{-6}$  & $4.2\cdot 10^{-5}$ & $0.4671$ & $0.9709$ \\
\hline
\end{tabular}
\end{table}
\end{example}

\begin{example}\rm\label{ex:ex3} The purpose of this example is to investigate the behavior in the presence of (nearly) multiple semisimple eigenvalues. We use $A_i=XD_iX^{-1}$ with $D_1={\rm diag}(1,1+\delta,1-\delta,2,2,2,3)$, 
$D_2={\rm diag}(1,1-\delta,1+\delta,1,2,3,3)$ for $\delta\ge 0$, and a random $X$ such that $\kappa_2(X)=10^4$.
For every combination of $\delta=10^{-4},10^{-6},\ldots,10^{-14}$ and noise level 
$\epsilon=10^{-4},10^{-6},\ldots,10^{-14}$, we compute the approximation errors
for the eigenvalue $\blambda^{(1)}=(1,1)$, which is not well-isolated when $\delta$ is small.
The results in Table \ref{table:Ex3} can be roughly clustered into the following three groups depending on
$(\epsilon,\delta)$:
\begin{enumerate}
\item[a)] For $\epsilon \le 10^{-4}\delta$, because of $\kappa_2(X)=10^{4}$, one can consider $\blambda^{(1)}$ a simple eigenvalue and apply Theorem~\ref{thm:semisimple} with $p=1$. 
Indeed, the errors in this setting behave similarly to the ones observed in 
Examples~\ref{ex:ex1} and~\ref{ex:ex2}. 
\item[b)]For $\epsilon = 10^{-2}\delta$ and $\epsilon = \delta$, $\blambda^{(1)}$ is
in a cluster of three joint eigenvalues. In this case, both $\epsilon$ and $\delta$ contribute to the error.
\item[c)]For $\epsilon\ge 10^2\delta$, the perturbation error dominates and 
one can consider $\blambda^{(1)}$ as a semisimple eigenvalue and apply Theorem~\ref{thm:semisimple} with $p=3$.
\end{enumerate}
As $\delta$ decreases, the empirical probability for ${\mathbb P}(b<a)$ is seen to decrease and is close to $0.5$ for pairs from group c). Still, also for these pairs 
${\mathbb P}(b<5a)$ remains very close to 1. In turn, the two-sided Rayleigh quotient approximation remains competitive also for tight clusters of semisimple eigenvalues.

\begin{table}[!hbt!]
\caption{Statistics of approximation errors for eigenvalue $\blambda_1$ from Example \ref{ex:ex3}.}
\label{table:Ex3}
\footnotesize
\begin{tabular}{|l|l|l|l|l|l|l|}
\hline \rule{0pt}{2.3ex}
& $\delta=10^{-4}$& $\delta=10^{-6}$ & $\delta=10^{-8}$ & $\delta=10^{-10}$ & $\delta=10^{-12}$ & $\delta=10^{-14}$ \\
\hline \rule{0pt}{2.3ex}%
 noise $\epsilon$ & \multicolumn{6}{c|}{Median of $a$ (error of one-sided Rayleigh quotient)} \\ 
\hline \rule{0pt}{2.3ex}%
$10^{-14}$ & $3.1\cdot 10^{-10}$ & $3.4\cdot 10^{-10}$ & $3.4\cdot 10^{-10}$ & $2.3\cdot 10^{-10}$ & $7.5\cdot 10^{-12}$ & $2.0\cdot 10^{-12}$ \\
$10^{-12}$ & $3.9\cdot 10^{-10}$ & $3.2\cdot 10^{-10}$ & $4.6\cdot 10^{-10}$ & $2.4\cdot 10^{-10}$ & $7.2\cdot 10^{-12}$ & $7.5\cdot 10^{-12}$ \\
$10^{-10}$ & $2.9\cdot 10^{-8}$  & $2.8\cdot 10^{-8}$  & $2.5\cdot 10^{-8}$  & $5.4\cdot 10^{-10}$ & $5.8\cdot 10^{-10}$ & $5.8\cdot 10^{-10}$ \\
$10^{-8}$  & $2.8\cdot 10^{-6}$  & $2.6\cdot 10^{-6}$  & $5.3\cdot 10^{-8}$  & $5.7\cdot 10^{-8}$  & $5.7\cdot 10^{-8}$  & $5.7\cdot 10^{-8}$  \\
$10^{-6}$  & $2.5\cdot 10^{-5}$  & $5.4\cdot 10^{-6}$  & $5.8\cdot 10^{-6}$  & $5.8\cdot 10^{-6}$  & $5.8\cdot 10^{-6}$  & $5.8\cdot 10^{-6}$  \\
\hline \rule{0pt}{2.3ex}%
 noise $\epsilon$ & \multicolumn{6}{c|}{Median of $b$ (error of two-sided Rayleigh quotient)} \\ 
\hline \rule{0pt}{2.3ex}%
$10^{-14}$ & $7.7\cdot 10^{-11}$ & $6.4\cdot 10^{-11}$ & $1.1\cdot 10^{-10}$ & $5.7\cdot 10^{-11}$ & $4.4\cdot 10^{-12}$ & $1.1\cdot 10^{-12}$ \\
$10^{-12}$ & $1.1\cdot 10^{-10}$ & $7.5\cdot 10^{-11}$ & $1.5\cdot 10^{-10}$ & $6.7\cdot 10^{-11}$ & $4.7\cdot 10^{-12}$ & $6.0\cdot 10^{-12}$ \\
$10^{-10}$ & $7.2\cdot 10^{-9}$  & $7.2\cdot 10^{-9}$  & $9.5\cdot 10^{-9}$  & $5.7\cdot 10^{-10}$ & $5.7\cdot 10^{-10}$ & $5.6\cdot 10^{-10}$ \\
$10^{-8}$  & $7.2\cdot 10^{-7}$  & $9.5\cdot 10^{-7}$  & $5.6\cdot 10^{-8}$  & $5.5\cdot 10^{-8}$  & $5.5\cdot 10^{-8}$  & $5.5\cdot 10^{-8}$  \\
$10^{-6}$  & $9.5\cdot 10^{-5}$  & $5.6\cdot 10^{-6}$  & $5.5\cdot 10^{-6}$  & $5.5\cdot 10^{-6}$  & $5.5\cdot 10^{-6}$  & $5.5\cdot 10^{-6}$  \\
\hline \rule{0pt}{2.3ex}%
 noise $\epsilon$ & \multicolumn{6}{c|}{empirical probability ${\mathbb P}(b<a)$} \\ 
\hline \rule{0pt}{2.3ex}%
$10^{-14}$ & $0.9495$ & $0.9712$ & $0.9334$ & $0.9524$ & $0.8039$ & $0.7068$ \\
$10^{-12}$ & $0.9474$ & $0.9506$ & $0.8922$ & $0.9403$ & $0.7069$ & $0.5855$ \\
$10^{-10}$ & $0.9736$ & $0.9735$ & $0.8373$ & $0.4561$ & $0.4817$ & $0.4952$ \\
$10^{-8}$  & $0.9699$ & $0.8446$ & $0.4436$ & $0.4902$ & $0.4902$ & $0.4903$ \\
$10^{-6}$  & $0.8413$ & $0.4558$ & $0.4997$ & $0.5005$ & $0.5005$ & $0.5005$ \\
\hline \rule{0pt}{2.3ex}%
 noise $\epsilon$ & \multicolumn{6}{c|}{empirical probability ${\mathbb P}(b<5a)$} \\ 
\hline \rule{0pt}{2.3ex}%
$10^{-14}$ & $0.9998$ & $0.9999$ & $0.9994$ & $0.9987$ & $0.9973$ & $0.9887$ \\
$10^{-12}$ & $1.9996$ & $0.9994$ & $0.9986$ & $0.9991$ & $0.9945$ & $0.9718$ \\
$10^{-10}$ & $1.0000$ & $1.0000$ & $0.9999$ & $0.9916$ & $0.9767$ & $0.9767$ \\
$10^{-8}$  & $1.0000$ & $1.0000$ & $0.9944$ & $0.9798$ & $0.9791$ & $0.9793$ \\
$10^{-6}$  & $1.0000$ & $0.9946$ & $0.9829$ & $0.9821$ & $0.9820$ & $0.9820$ \\
\hline
\end{tabular}
\end{table}

Figure \ref{fig:Ex3_Fig} shows the error distribution of the errors. For $\delta=10^{-6}$, the curves look similar to the ones before for well-separated simple eigenvalues when $\epsilon\le 10^{-10}$. The curves for $\epsilon=10^{-8}$ and $\epsilon=10^{-6}$ are shaped differently because $\epsilon$ is too large and $\blambda^{(1)}$ does not behave
 anymore as a simple eigenvalue. For $\delta=10^{-12}$, all curves have this shape.
 We checked that this behaviour continues down to $\delta=0$. In summary, it can be observed that multiple semisimple eigenvalues are computed with roughly the
 same accuracy as simple eigenvalues.
   
\begin{figure}[htb!]
    \centering
 \includegraphics[width=6.4cm]{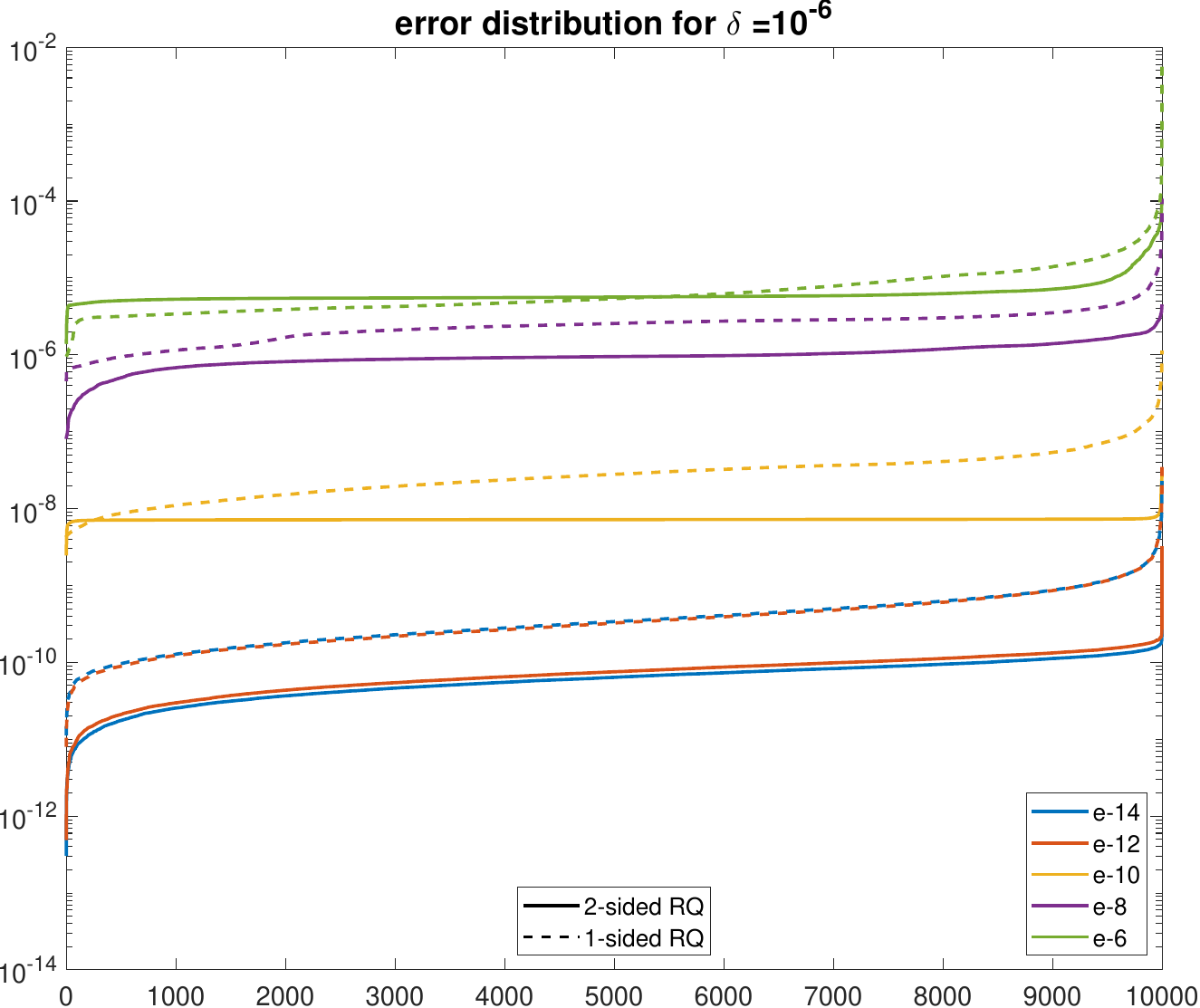}\ \
 \includegraphics[width=6.4cm]{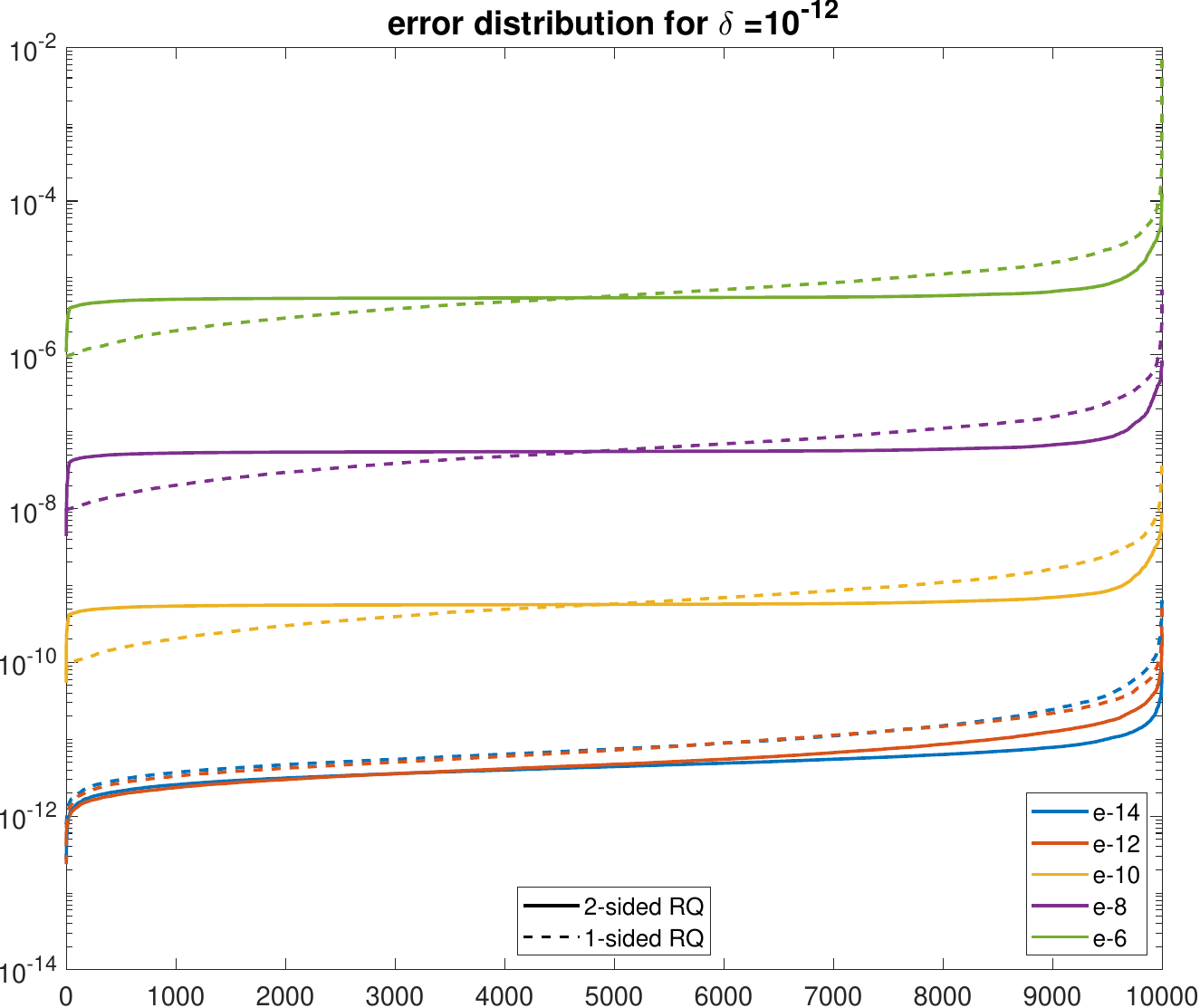}
    \caption{Distribution of absolute errors for eigenvalue $\blambda^{(1)}$ from 
    Example \ref{ex:ex3} for $\delta=10^{-6}$ (left) and $\delta=10^{-12}$ (right), using one-sided (dashed lines) and two-sided (solid lines) Rayleigh quotients.}
    \label{fig:Ex3_Fig}  
\end{figure}

\end{example}

\begin{example}\rm\label{ex:ex5}
    We now consider $A_i=XB_iX^{-1}$ with random $X$, such that $\kappa_2(X)=10$, and
    \[B_1=\left[\begin{matrix}1 & 1 & & & & \cr & 1 & 1 & & & \cr & & 1 & & & \cr & & & 2 & & \cr & & & & 3 & \cr & & & & & 4\end{matrix}\right],\quad
    B_2=\left[\begin{matrix}1 & 1 & & & & \cr & 1 & 1 & & & \cr & & 1 & & & \cr & & & 4 & & \cr & & & & 3 & \cr & & & & & 2\end{matrix}\right].\]
    In exact arithmetic, the eigenvalue $\blambda^{(1)}=(1,1)$  is multiple and \emph{not} semisimple. In the presence of roundoff error, this multiplicity is broken but we can expect the condition numbers of the eigenvalue $\blambda^{(1)}$  and the eigenvector matrix 
    to become very high. In particular, the results of Theorem~\ref{thm:semisimple} do not apply in this situation.
    
    As $\widetilde A(\mu)$ is diagonalizable, one can still attempt to apply Algorithm~\ref{alg:RJEA}. Figure~\ref{fig:Ex5_Fig} shows that the error for the triple eigenvalue $\blambda^{(1)}$ is on the level of $\epsilon^{1/3}$, the same level of error that can be expected when perturbing an eigenvalue 
    in a Jordan block of size three in a single matrix.
    On the other hand, the error for $\blambda^{(4)} = (2,4)$
    behaves in the same way as the errors for simple eigenvalues in the other examples. 

\begin{figure}[htb!]
    \centering
 \includegraphics[width=6.4cm]{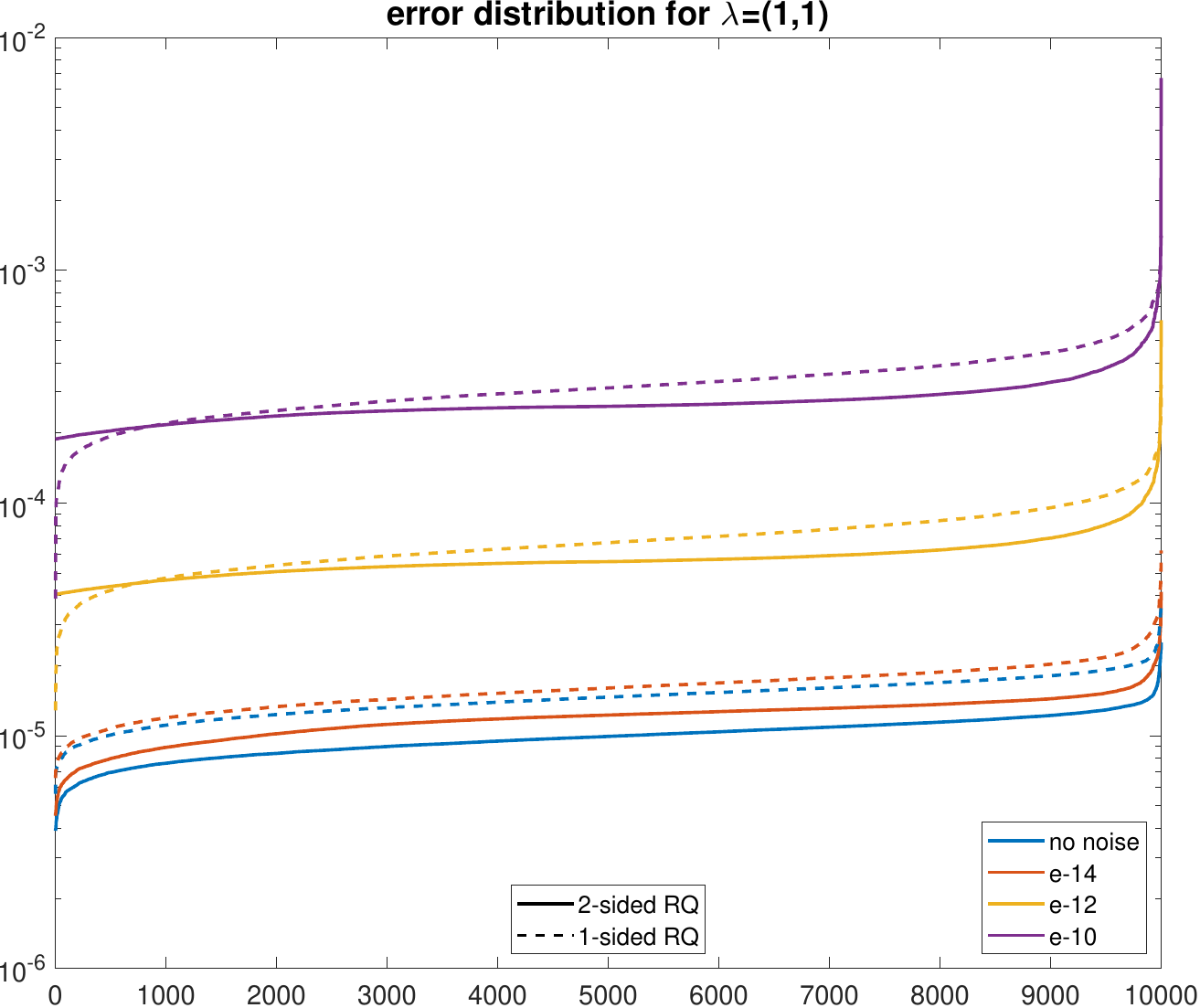}\ \
 \includegraphics[width=6.4cm]{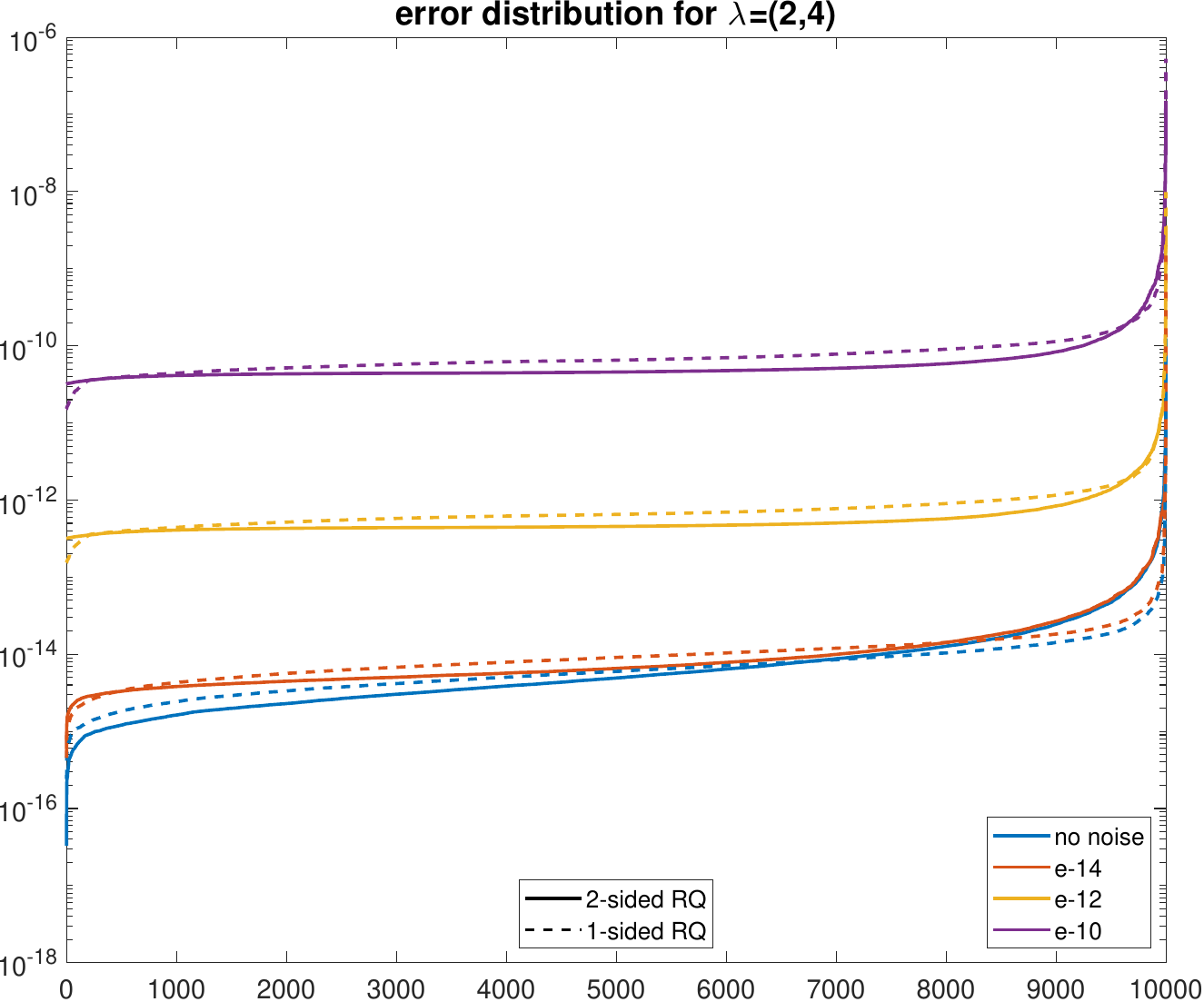}
    \caption{Distribution of absolute errors for eigenvalues $\blambda^{(1)}$ (left) and $\blambda^{(4)}$ (right) from Example \ref{ex:ex5}, using one-sided (dashed lines) and two-sided (solid lines) Rayleigh quotients.}
    \label{fig:Ex5_Fig}  
\end{figure}
\end{example}

\section{Multiparameter eigenvalue problems}\label{sec:mep}

This section is concerned with one of the applications that motivated this work.
A \emph{$d$-parameter eigenvalue problem} has the form
\begin{equation}\label{eq:mep}
    A_{i0}x_i = \lambda_1 A_{i1}x_i + \cdots + \lambda_d A_{id}x_i, \quad i=1,\ldots,d,
\end{equation}
where $A_{ij}$ is an $n_i\times n_i$ complex matrix and $x_i\ne 0$ for $i=1,\ldots,d$.
When~\eqref{eq:mep} is satisfied,
a $d$-tuple $\blambda=(\lambda_1,\ldots,\lambda_d)\in\CC^d$ is called an \emph{eigenvalue} and 
$x_1\otimes \cdots \otimes x_d$ is a corresponding \emph{eigenvector}. 
Generically, a multiparameter eigenvalue problem~\eqref{eq:mep} has $N=n_1\cdots n_d$ eigenvalues that
are roots of a system of $d$ multivariate characteristic polynomials
$p_i( \lambda_1,\ldots,\lambda_d):=\det(A_{i0}-\lambda_1 A_{i1} -\cdots -\lambda_d A_{id})=0$
for $i=1,\ldots,d$. 

The problem~\eqref{eq:mep}
is closely related to a system of $d$ generalized eigenvalue problems 
\begin{equation}\label{eq:Delta}
    \Delta_i z =\lambda_i \Delta_0z, \quad i=1,\ldots,d,
\end{equation}
with $z=x_1\otimes\cdots\otimes x_d$ and the $N\times N$ matrices
$$\Delta_0=\left|\begin{matrix}A_{11} & \cdots & A_{1d}\cr
\vdots &  & \vdots \cr 
A_{d1} & \cdots & A_{dd}\end{matrix}\right|_\otimes
=\sum_{\sigma\in S_d}{\rm sgn}(\sigma) \, A_{1\sigma_1}\otimes A_{2\sigma_2}\otimes \cdots \otimes A_{d\sigma_d},
$$
$$\Delta_i=\left|\begin{matrix}A_{11} & \cdots & A_{1,i-1} & A_{10} & A_{1,i+1} & \cdots & A_{1d}\cr
\vdots &  & \vdots & \vdots & \vdots & & \vdots \cr 
A_{d1} & \cdots & A_{d,i-1} & A_{d0} & A_{d,i+1} & \cdots & A_{dd}\end{matrix}\right|_\otimes,\quad i=1,\ldots,d,
$$
which are called \emph{operator determinants}~\cite{AtkinsonBook}.
In the following we will assume that
$\Delta_0$ is invertible, in which case~\eqref{eq:mep} is called \emph{regular} and the matrices 
$\Gamma_i:=\Delta_0^{-1}\Delta_i$ for $i=1,\ldots,d$ commute. 
If $N$ is not too large, 
then a standard approach to solve \eqref{eq:mep} is to explicitly compute the matrices $\Gamma_1,\ldots,\Gamma_d$ and then solve the joint eigenvalue problem, see, e.g., \cite{SlivnikTomsic, HKP_JD2EP}.
In the following, we will discuss randomization-based methods.
\smallskip

\textbf{The right-definite case:}
If all matrices $A_{ij}$ defining~\eqref{eq:mep} are real and symmetric, and  $\Delta_0$ is positive definite, the problem~\eqref{eq:mep} is called \emph{right-definite}. This allows us to perform the Cholesky decomposition $\Delta_0=VV^T$ and transform~\eqref{eq:Delta} into the joint eigenvalue problem
$D_iw=\lambda_i w$, $i=1,\ldots,d$, where $D_i=V^{-1}\Delta_i V^{-T}$ is symmetric and $w=V^Tz$. In exact arithmetic, the matrices $D_1,\ldots,D_d$ commute and can be 
simultaneously diagonalized by an orthogonal matrix $Q$; see also~\cite{SlivnikTomsic}.
In turn, the right-definite case allows one to apply the randomized joint diagonalization method from~\cite{HeKressner}, which diagonalizes a random linear combination $D(\mu)$ to produce the orthogonal transformation $Q$. In~\cite{HeKressner}, this approach has been shown to be robust to noise. When $\Delta_0$ is ill-conditioned, the recent randomized approach for simultaneous diagonalization by congruence in \cite{HeKressnerRSDC} can be more reliable.
\smallskip

\textbf{The non-definite regular case:}
In many applications, $\Delta_0$ is invertible but not symmetric definite. As the matrices $\Gamma_i=\Delta_0^{-1}\Delta_i$ commute, they admit a simultaneous Schur factorization, that is, they can be triangularized simultaneously by the same unitary similarity transformation. The matrix pencils
$(\Delta_1,\Delta_0),\ldots,(\Delta_d,\Delta_0)$ admit a simultaneous generalized Schur factorization, that is, there are unitary matrices $Q$ and $Z$ such that every
$Q\Delta_iZ$ is upper triangular for $i=0,\ldots,d$. The approach in~\cite[Algorithm 2.3]{HKP_JD2EP} for two-parameter eigenvalue problems first applies the QZ algorithm to the pencil $(\Delta_1,\Delta_0)$, reorders the generalized Schur form so that (nearly) multiple eigenvalues are grouped together, and then applies $Q$ and $Z$ to $\Delta_2$. This yields block upper triangular matrices and the eigenvalues are extracted from the diagonal blocks of $Q \Delta_0 Z$, $Q \Delta_1 Z$, $Q \Delta_2 Z$. A problem with this approach is that the correct reordering of the generalized Schur form is difficult to ensure in the presence of noise. Also, this procedure has to be repeatedly applied for problems with more than two parameters.

An alternative to Schur forms is to use eigenvectors.
Let us take a random linear combination 
$\Delta(\mu) = \sum_{i=1}^d \mu_i \Delta_i$.
Then we can compute the left and right eigenvector matrices either of $\Delta_0^{-1}\Delta(\mu)$ or of the pencil $(\Delta(\mu),\Delta_0)$, and apply the resulting transformation to the other generalized eigenvalue
problems from~\eqref{eq:Delta}. The use of pencils avoids the multiplication with $\Delta_0^{-1}$ and corresponds to the use of generalized Rayleigh quotients.
Specifically, the \emph{generalized one-sided Rayleigh quotients} take the form
\[
\rho(x_i,\Delta_k,\Delta_0) = \frac{x_i^*\Delta_k x_i}{x_i^* \Delta_0 x_i},\quad k=1,\ldots,d,\ i=1,\ldots,N,
\]
while the \emph{generalized two-sided Rayleigh quotients} are given by
\[
\rho(x_i,y_i,\Delta_k,\Delta_0) = \frac{y_i^*\Delta_k x_i}{y_i^* \Delta_0 x_i},\quad k=1,\ldots,d,\ i=1,\ldots,N.
\]

\textbf{The singular case:}
 It can happen that every linear combination of the matrices $\Delta_0,\Delta_1,\ldots,\Delta_d$ 
is singular, but~\eqref{eq:mep} still has finitely many solutions that are roots of $d$ multivariate characteristic polynomials.
Such problems are called \emph{singular multiparameter eigenvalue problems}, resulting in a system of $d$ singular matrix pencils~\eqref{eq:Delta}. The classical approach to deal with a singular pencil proceeds by extracting the regular part by the 
staircase algorithm \cite{VD79} and using the QZ algorithm.
Recently, randomized methods for singular pencils have been proposed that do not require the  use of the staircase method \cite{HMP_SingGep, HMP_SingGep2, DanielIvana}. Unfortunately, none of these methods has been so far 
generalized to a family of singular pencils~\eqref{eq:Delta}. The only numerical method
 available is a generalized staircase-type algorithm \cite{MP_Q2EP}. The regular problem extracted by this method can be addressed with the algorithm proposed in this work.

\subsection{Solution of multiparameter eigenvalue problems with Algorithm~\ref{alg:RJEA}}

We suggest to solve a multiparameter eigenvalue problem in the following way. First, compute eigenvalues and, more importantly, eigenvectors 
of $\Gamma(\mu) = \Delta_0^{-1}\Delta(\mu)$, where $\Delta(\mu)$ is a random linear combination of matrices $\Delta_1,\ldots,\Delta_d$. For each right and left eigenvector $z$ and $w$ we can then compute the corresponding eigenvalue
$(\lambda_1,\ldots,\lambda_d)$ from two-sided Rayleigh quotients because $z$ and $w$ are common eigenvectors for all matrices $\Gamma_1,\ldots,\Gamma_d$.

We note that for the above we do not have to compute all operator determinants $\Delta_0,\ldots,\Delta_d$ explicitly. This is important because these matrices can be very large and consume a lot of memory. 
We explicitly need just $\Delta_0$ and $\Delta(\mu)$, which we can compute from the matrices in \eqref{eq:mep} and values $\mu_1,\ldots,\mu_d$. 
Once we have an eigenvector $z$ we can 
think of it as a vectorization of a $d$-dimensional tensor ${\cal Z}\in\CC^{n_1\times \cdots \times n_d}$, i.e., $z={\rm vec}({\cal Z})$. Multiplication by $\Delta_0$ (and similarly with $\Delta_1,\ldots,\Delta_d$) in the 
computation of Rayleigh quotients can be then efficiently performed 
as $\Delta_0 z ={\rm vec}({\cal W})$, where
\[{\cal W} = \sum_{\sigma\in S_d}{\rm sgn}(\sigma) \, A_{1\sigma_1}\times_1 A_{2\sigma_2}\times_2 \cdots A_{d\sigma_d} \times_d {\cal Z}\]
and $\times_1,\ldots,\times_d$ are multiplications of tensor by matrices in directions $1,\ldots,d$, for details see, e.g., \cite{KoldaBader_Tensors}.

\subsection{Numerical examples}

\begin{example}\rm\label{ex:ex7}
Recently~\cite{Eisenmann24}, it was exposed that the numerical solver for multiparameter eigenvalue problems in \cite{multipareig_2023} sometimes fails to find all eigenvalues of random three-parameter eigenvalue problems. 
The observed failures can be attributed to
the method for computing joint eigenvalues in~\cite{multipareig_2023}, which is 
based on the generalized Schur form followed by a simple clustering. In the new release~\cite{multipareig_28} 
this is replaced by the method based on Algorithm \ref{alg:RJEA} from
this paper. 
The new solver is much faster and computes eigenvalues of the problems reported in \cite{Eisenmann24} more accurately.

To demonstrate this, we constructed three-parameter eigenvalue problems 
in a similar way as in \cite{Eisenmann24}. We take $n\times n$ matrices 
\begin{equation}\label{eq:eisenmann}
  A_{ij}=Q_{ij}D_{ij}Q_{ij}^T+\delta_{ij}I_n,\quad i=1,2,3,\ j=0,1,2,3,
\end{equation}  
where $Q_{ij}$ is a random orthogonal matrix (constructed from the QR decomposition of
a Gaussian random matrix), $D_{ij}$ is a diagonal matrix whose diagonal elements
are random values uniformly distributed in $[-\frac{1}{2n},\frac{1}{2n}]$, and $\delta_{ij}$ is the Kronecker delta.
For each size $n=4,\ldots,16$ we generated 10 different problems of form
\eqref{eq:eisenmann} and for each computed the largest error of obtained eigenpairs. We compared 
{\tt threepareig}, which is a Matlab solver for three-parameter eigenvalue problems in MultiParEig, from 
\cite{multipareig_2023} and the updated version \cite{multipareig_28}.
The results are presented in Figure \ref{fig:Ex7_Fig}. If we compare the two-sided (2S) and one-sided (1S) versions, we see that the two-sided version is slightly more expensive but much more accurate. 

\begin{figure}[htb!]
    \centering
 \includegraphics[width=6.4cm]{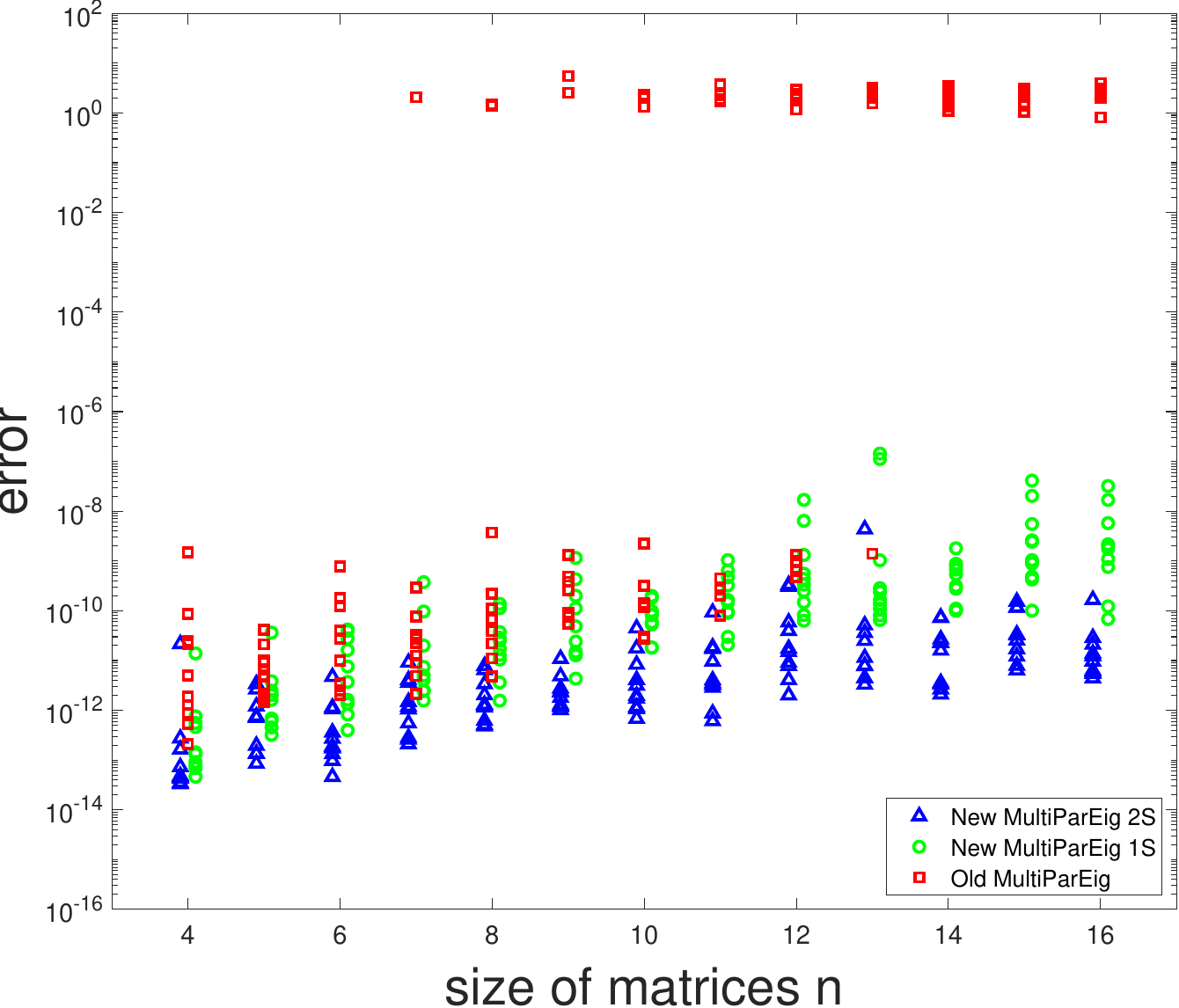}\ \
 \includegraphics[width=6.4cm]{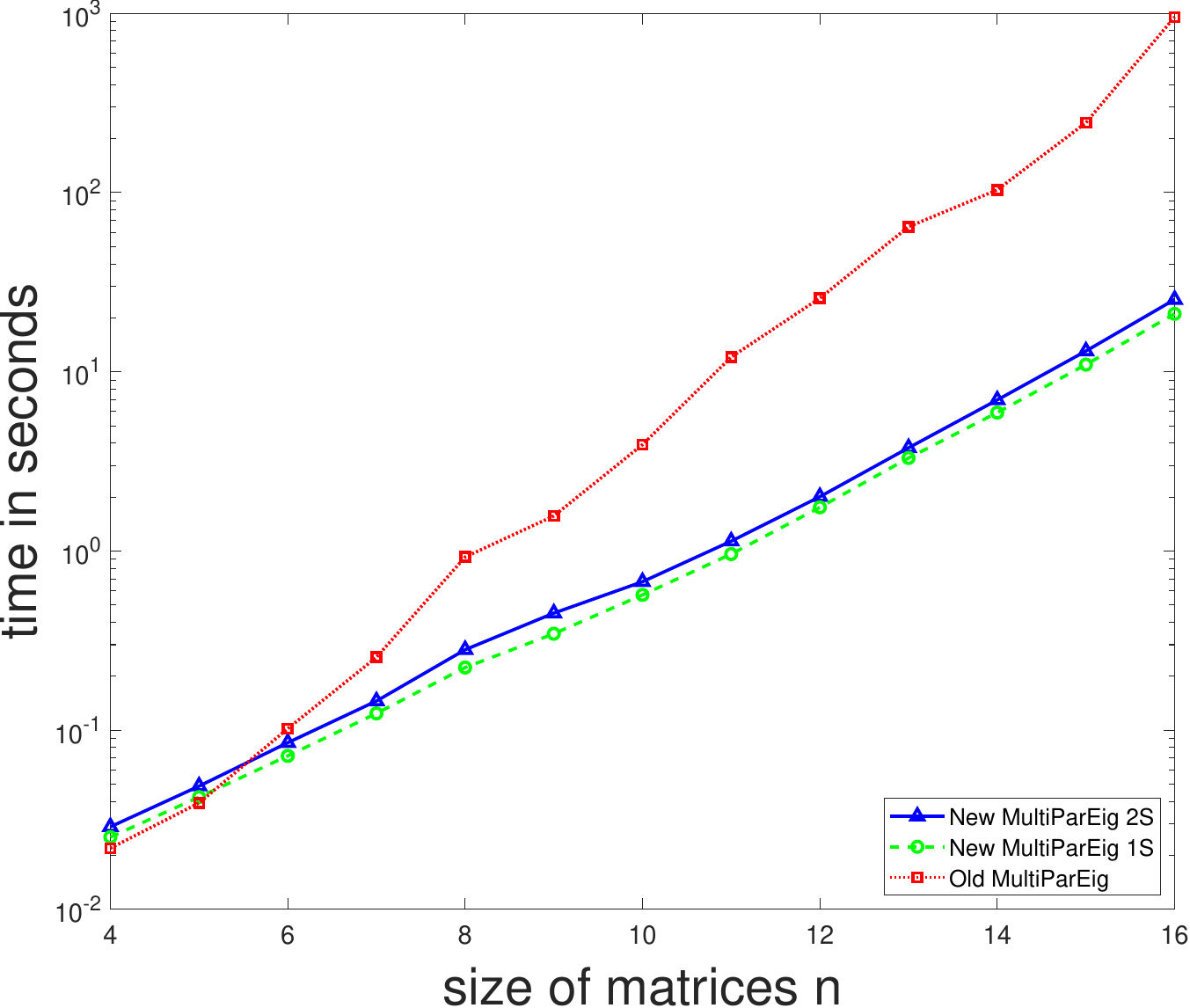}
    \caption{Comparison of old and new solver for three-parameter eigenvalue problems from \cite{multipareig_2023}
    and \cite{multipareig_28} on problems from Example \ref{ex:ex7}. The left figure presents the maximal error of computed eigenpairs, while the right figure compares computational times.}
    \label{fig:Ex7_Fig}  
\end{figure}

\end{example}

\begin{example}\rm\label{ex:ex6} Wave propagation in an elastic layer that
is in contact with another medium of infinite extent at one of the surfaces can be modelled by 
a nonlinear eigenvalue problem of the form
\begin{equation}\label{eq:NEP_Leaky}
(-k^2E_0+(\textrm{i} k)E_1-E_2+\omega^2 M+\sqrt{\kappa_1^2-k^2}R_1+\sqrt{\kappa_2^2-k^2}R_2)u=0,
\end{equation}
where $E_0,E_1,E_2,M,R_1,R_2$ are $n\times n$ matrices,
whose solutions $(\omega,k)$ give dispersion curves, where $\omega$ is the frequency and $k$ is the wavenumber.

Recently \cite{GPKJ_Embedded_Leaky}, it was shown that for a given $\omega$ we can get the corresponding solutions $k$ of \eqref{eq:NEP_Leaky} by solving a four-parameter eigenvalue problem 
    \begin{align*}
      \left(-E_2+\omega^2 M + \textrm{i} k E_1 + \textrm{i} \eta_1 R_1 + \textrm{i} \eta_2 R_2 + \xi E_0\right)u & =0\\
      \left(\left[\begin{matrix} 0 & -\eta_1^2 \cr 1 & 0 \end{matrix}\right]
      + \textrm{i} \eta_1 \left[\begin{matrix} 1 & 0 \cr 0 & 1 \end{matrix}\right]
      + \xi \left[\begin{matrix} 0 & -1 \cr 0 & 0 \end{matrix}\right]\right) x_1                                          & = 0 \\
      \left(\left[\begin{matrix} 0 & -\eta_2^2 \cr 1 & 0 \end{matrix}\right]
      + \textrm{i} \eta_2 \left[\begin{matrix} 1 & 0 \cr 0 & 1 \end{matrix}\right]
      + \xi \left[\begin{matrix} 0 & -1 \cr 0 & 0 \end{matrix}\right]\right) x_2                                          & = 0 \\
      \left(\left[\begin{matrix} 0 & 0 \cr 0 & 1 \end{matrix}\right]
      + \textrm{i} k \left[\begin{matrix} 0 & 1 \cr 1 & 0 \end{matrix}\right]
      + \xi \left[\begin{matrix} 1 & 0 \cr 0 & 0 \end{matrix}\right]\right) x_3                                          & = 0, 
    \end{align*}
  where $\eta^2_1=\kappa_1 - k^2$, $\eta^2_2=\kappa_2^2 - k^2$, $\xi=-k^2$, 
  and $x_1,x_2,x_3\neq 0$. This leads to $\Delta$-matrices
  of size $8n\times 8n$ such that $\Delta_0^{-1}\Delta_1,\ldots,\Delta_0^{-1}\Delta_4$
  commute, for more details, see \cite{GPKJ_Embedded_Leaky}.

  For a particular problem of a brass plate coupled to infinite Teflon, using the same 
  matrices of size $45\times 45$ in \eqref{eq:NEP_Leaky} as in \cite[Section 6.2]{GPKJ_Embedded_Leaky}, 
  we computed the points on the dispersion curves
  for $150$ frequencies up to $3$Mhz using the old solver for multiparameter eigenvalue problems 
  in \cite{multipareig_2023} and the new solver in \cite{multipareig_28} that uses two-sided Rayleigh quotients from eigenvectors 
  of a random linear combination $\Delta_0^{-1}(\mu_1 \Delta_1+\cdots+\mu_4\Delta_4)$. 
  The comparison is in Figure \ref{fig:Leaky}, where
  we see that the old algorithm could not compute all the points on the two horizontal dispersion curves accurately.
  
  \begin{figure}[htb!]
    \centering
 \includegraphics[width=6.4cm]{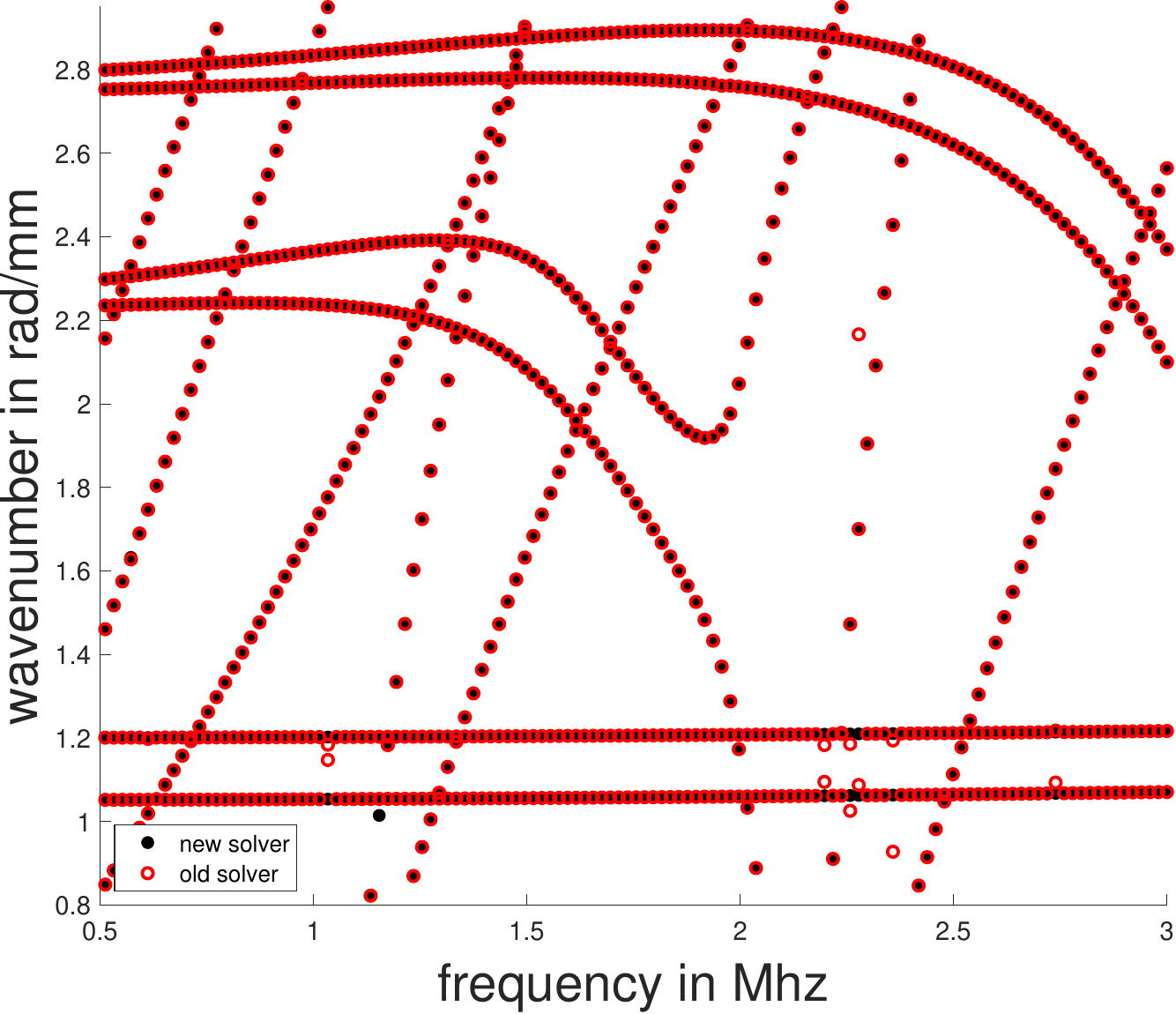}
    \caption{Points on dispersion curves related to the problem of wave propagation of a 
    brass plate coupled to infinite Teflon from \cite{GPKJ_Embedded_Leaky}, computed
    by the old and updated solver for multiparameter eigenvalue problems.}
    \label{fig:Leaky}  
\end{figure}

\end{example}


\section{Roots of polynomial systems}\label{sec:poly_roots}

Another important source of problems leading to joint eigenvalue problems for a commuting family of matrices are eigenvector-based methods for solving multivariate polynomial equations; see, e.g., \cite{CorlessGianniTrager97,MollerStetter_NumMath, Stetter_SIAMBook}.
In this section, we provide a brief introduction to such an approach and demonstrate the accuracy of the proposed two-sided Rayleigh quotient method for joint eigenvalue problems that arise in polynomial root finding.

Let $\mathcal{P}^s$ denote the set of polynomials over $\mathbb{C}$ in $s$ variables $x_1,\ldots,x_s$. Given a system $P$ of $n$ polynomials $p_i \in \mathcal{P}^s, i = 1,\ldots,n$, a fundamental task of computer algebra and numerical polynomial algebra is to compute their common roots, or the so-called algebraic variety $V:= \{x \in \mathbb{C}^s: p_i(x) = 0, i=1,\ldots,n\}$ of $P$. 
For this purpose,
we need to study the vector space $\mathcal{R}[\mathcal{I}] := \mathcal{P}^s / \mathcal{I}$ of the \emph{quotient ring} or the \emph{residue class ring modulo the ideal} $\mathcal{I}$ generated by $P$, which 
is defined as:
\[\mathcal{I} = \langle p_1,\ldots,p_n \rangle := \{ c_ip_i: c_i \in \mathcal{P}^s \}.\]
 When the algebraic variety $V$ is $0$-dimensional, i.e., $V$ only consists of finitely many points, the vector space $\mathcal{R}[\mathcal{I}]$ has the dimension equal to the number of common roots (counting their multiplicities). This allows us to perform computation of linear algebra if a basis $\mathcal{B}$ of  $\mathcal{R}[\mathcal{I}]$ is provided. Let $M_g$ denote the matrix associated with the linear operation of polynomial multiplication with a given polynomial $g \in \mathcal{P}^s$ in the vector space $\mathcal{R}[\mathcal{I}]$ with respect to the basis $\mathcal{B}$. It turns out that the family $\{M_g: g \in \mathcal{P}^s\}$ 
 of all the possible multiplication matrices is a commuting family since polynomial multiplication is commutative \cite{Stetter_SIAMBook}. For illustration purposes, we assume all the roots $z_1,\ldots,z_m$ are simple. It can be shown that the joint eigenvectors for this commuting family are the Lagrange polynomials, which evaluate to $1$ on exactly one $z_i$ and $0$ on the other $z_j,j\neq i$, and that the corresponding eigenvalue for the matrix $M_g$ is $g(z_i)$.  Refer to \cite[Proposition 6.1]{grafTownsend24} for further details on Lagrange polynomials defined on the common roots of a given polynomial system.  Hence, the joint eigenvalues of the commuting family 
\[\mathcal{M} := \{M_{x_1},\ldots,M_{x_s}\}\]
are the common roots~\cite[Theorem 2.27]{Stetter_SIAMBook}. These multiplication matrices can be obtained through symbolic computation of the Gr\"{o}bner basis~\cite{GrobnerEVP}, direct computation of the normal forms from Macaulay matrices~\cite{Telen18}, and by leveraging the shift-invariant structure of the numerical basis of the null space of Macaulay matrices~\cite{Vermeerschthesis}.
For further details on numerical polynomial algebra, see, e.g., the classical textbook~\cite{Stetter_SIAMBook}.

We note that a similar approach with block Macaulay matrices that also requires joint eigenvalues 
of commuting matrices, was recently developed for rectangular multiparameter eigenvalue problems, for details see, e.g., \cite{VermeerschDeMoorLAA}. A rectangular multiparameter eigenvalue problem is 
different from \eqref{eq:mep}, a generic linear form is 
\begin{equation}\label{rectmep}
  B_{0} x = \lambda_1 B_{1} x + \cdots + \lambda_d B_{d} x,
\end{equation}
where $B_0,\ldots,B_d$ are rectangular matrices of size $(n+d-1)\times n$. If
$\lambda_1,\ldots,\lambda_d$ satisfy \eqref{rectmep} for a nonzero $x$, then 
$\blambda=(\lambda_1,\ldots,\lambda_d)$ is an eigenvalue. It is also possible to 
solve 
\eqref{rectmep} with tools for standard multiparameter eigenvalue problems of the form \eqref{eq:mep}, for 
details see \cite{HKP_NLAA_2023_RectMEP}.

\begin{example}\rm\label{ex:poly_sys} As we have already observed the advantage of the two-sided Rayleigh quotient over the one-sided one in Section \ref{sec:synthetic}, we only compare the two-sided Rayleigh quotient with Schur decomposition of one random linear combination ({\tt RSchur}) used in~\cite{Telen18} and Schur decomposition of the first matrix ({\tt Schur}) used in \cite{Vermeerschthesis}. Our comparison focuses on the commuting families generated from random polynomial systems of $3$ variables with maximum degree $10$ (resulting in $3$ $1000 \times 1000$ matrices named {\tt random})\footnote{We use the implementation in \cite{Telen18}, available at \url{https://github.com/simontelen/NormalForms}, to get the multiplication matrices.}, the polynomial system arising in the computation of general economic equilibrium models~\cite{rose} (resulting in $4$ $136 \times 136$ matrices named {\tt rose}), the  polynomial system in magnetism in physics~\cite{katsura} ( resulting in $7$ $128 \times 128$ matrices named {\tt katsura7}) and another system in economics~\cite[p. 148]{Morgan09} (resulting in $9$ $64 \times 64$ matrices named {\tt redeco8}).\footnote{Examples {\tt rose}, {\tt katsura7} and {\tt redeco8} are documented in PHCpack~\cite{PHCpack} and matlab codes for these examples are documented in~\cite{Vermeerschthesis}, available at \url{https://gitlab.esat.kuleuven.be/Christof.Vermeersch/macaulaylab-public}.} The commuting family in {\tt random} is generated using normal forms from Macaulay matrices~\cite{Telen18}, while 
in {\tt rose}, {\tt katsura7} and {\tt redeco8} 
the null space of Macaulay matrices~\cite{Vermeerschthesis} is used.  The error is measured by the residual in terms of the roots of the corresponding polynomial systems, averaged over $100$ runs. Note that in the original implementations in \cite{Vermeerschthesis} and \cite{Telen18}, there is an option for eigenvalue clustering. However, we observed significant errors due to clustering for {\tt RSchur}, so we excluded those results from our comparison. We also observed that generating the multiplication matrices is significantly more expensive than computing the joint eigenvalues. Therefore, only a comparison of the accuracy is included.

The results presented in Table \ref{table:polynomials} illustrate that the two-sided Rayleigh quotient method can achieve superior accuracy for both synthetic polynomial systems and polynomial systems from real-world applications. 
\begin{table}[!hbt!]
\caption{Polynomial system accuracy comparison}

\label{table:polynomials}
\small
\begin{tabular}{|c|S[table-format=1.2e3]|S[table-format=1.2e3]|S[table-format=1.2e3]|S[table-format=1.2e3]|}
\hline
 & {random} & {rose} &  {katsura7} & {redeco8}\\ 
\hline
Rschur & 1.31e-11 &  1.98e-08 & 1.74e-10 &3.09e-12\\
schur & 5.58e-12 & 1.12e-08 & 8.19e-10 &2.75e-12\\
RQ2 & 9.72e-14 & 9.26e-09 & 1.61e-11  &3.52e-12\\
\hline
\end{tabular}
\end{table}
\end{example}

\begin{example}\rm\label{ex:ex8}
It is well known that we can compute roots of a univariate polynomial as eigenvalues of its companion matrix. 
In a similar way it is possible to linearize a system of two bivariate polynomials into
a two-parameter eigenvalue problem whose eigenvalues agree with the roots, see, e.g., \cite{PH_BiRoots_SISC}. 
We start from a 
system of two bivariate polynomials $p_i(x_1,x_2)=0$, $i=1,2$, and construct matrices $A_{ij}$
of a so-called \emph{determinantal representation} such that
$\det(A_{i0}-x_1 A_{i1}-x_2 A_{i2})=p_i(x_1,x_2)$. The eigenvalues of the obtained two-parameter eigenvalue
problem then give the roots of the initial polynomial system.

Recently \cite{grafTownsend24}, the stability of 
this approach and of several other methods from Section \ref{sec:poly_roots} was analyzed. Although the main result in \cite{grafTownsend24} is pessimistic and shows that such methods for computing roots of systems of bivariate polynomials are unstable,
we can show that by replacing the old solver for multiparameter eigenvalue problems with the new two-sided Rayleigh quotient method we get more accurate results. The results are presented in
Figure \ref{fig:Ex8_Fig}, where we recreate the example in 
\cite[Figure 1.(e)]{grafTownsend24}.

The polynomials in this example have the form
\begin{equation}\label{eq:GrafTownsend} p_i(x_1,x_2)=(x_i-\textstyle{\frac{1}{3}})^2 + \sigma \sum_{j=1}^2q_{ij}(x_j-\textstyle{\frac{1}{3}}),\quad i=1,2,
\end{equation}
where $\sigma>0$ and $q_{ij}$ are entries of a random $2\times 2$  real orthogonal matrix.
The system has a root at $(\frac{1}{3},\frac{1}{3})$.
Like in \cite{grafTownsend24}, we computed roots of 1000 random systems of the form
\eqref{eq:GrafTownsend} for each value of $\sigma$ and we report the median error of the
root $(\frac{1}{3},\frac{1}{3})$. 
As shown in \cite{grafTownsend24}, for small $\sigma>0$, the condition number of $(\frac{1}{3},\frac{1}{3})$ as a root of a polynomial system is ${\cal O}(\sigma^{-1})$, 
while the condition number of $(\frac{1}{3},\frac{1}{3})$ as an eigenvalue
of the related two-parameter eigenvalue problem is ${\cal O}(\sigma^{-2})$. Because of that,
while the accuracy of a stable method for computing roots of bivariate polynomials should be around the blue line on Figure \ref{fig:Ex8_Fig}, the accuracy of all methods from Section \ref{sec:poly_roots} is expected to be much worse and close to the red line. We see that the new solver gives much more accurate results that are even close to the blue line of stability if $\sigma$ is not too small.

\begin{figure}[htb!]
    \centering
 \includegraphics[width=6.4cm]{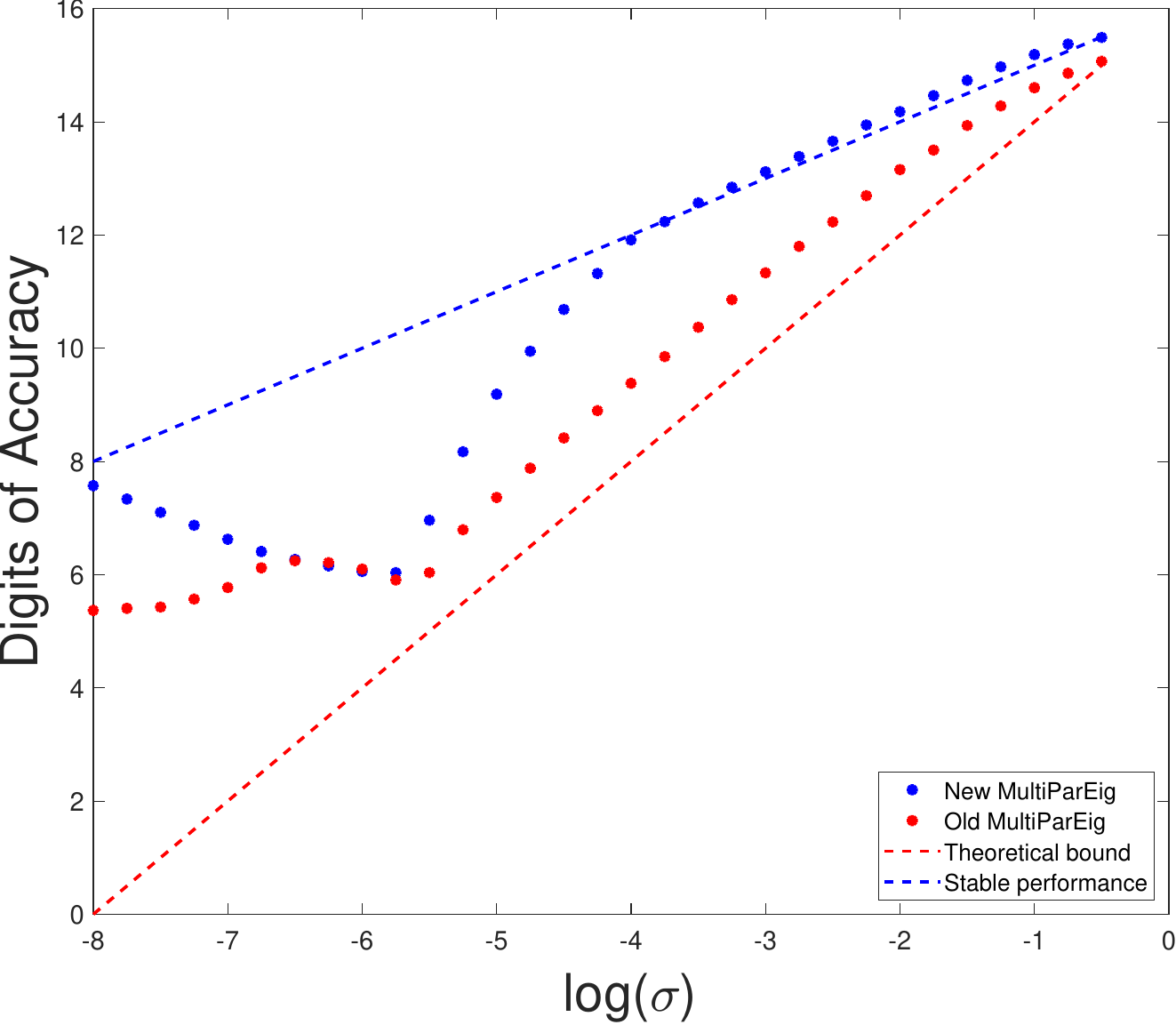}\ \
    \caption{Comparison of accuracy obtained with the old and new solver for two-parameter eigenvalue problems applied to 
    polynomials in \cite[Example 1.(e)]{grafTownsend24}.}
    \label{fig:Ex8_Fig}  
\end{figure}

\end{example}

\section{Conclusions}\label{sec:conc}

We proposed a simple numerical approach to compute joint eigenvalues of a family of (nearly) commuting matrices, combining eigenvectors of a random linear combination with  Rayleigh quotients. Our main results show that this approach, in particular the use of two-sided Rayleigh quotients, accurately computes well-conditioned 
semisimple joint eigenvalues with high probability. It still works satisfactorily in the presence of defective eigenvalues. 
Numerous numerical examples show that the method can be efficiently used in solvers for multiparameter 
eigenvalue problems and roots of systems of multivariate polynomials. The method is included in the new release
of MultiParEig, a Matlab toolbox for multiparameter eigenvalue problems \cite{multipareig_28}.
\medskip

\noindent\textbf{Acknowledgements}\quad BP would like to thank Hauke Gravenkamp and Daniel A. Kiefer for providing data for  
Example \ref{ex:ex6}.
\smallskip










\appendix
\section{Perturbations of invariant subspaces}

This appendix contains the eigenspace perturbation expansions needed for the analysis in Section~\ref{subsec:structural_bound}. These expansions are direct consequences of existing perturbation results in, e.g.,~\cite{KarowKressner2014} and~\cite[Section 5]{StewartSun}.

\begin{lemma}
\label{lemma:first_order_pert_subspace}
    Consider a matrix $A \in \mathbb{C}^{n \times n}$ transformed to a block diagonal form 
    \[ Y^*AX = \begin{bmatrix}
       A_{11} & 0 \\
        0 & A_{22}
    \end{bmatrix},\quad A_{11}\in\mathbb{C}^{p\times p},\ A_{22}\in\mathbb{C}^{(n-p)\times (n-p)},\]   
    such that $Y^* X =I$ and the spectra of $A_{11}$ and $A_{22}$ are disjoint. Partition
    $X= [X_1\ X_2] \in \mathbb{C}^{n \times n}$ and $Y= [Y_1\ Y_2] \in  \mathbb{C}^{n \times n}$
    accordingly. Then, given a perturbed matrix $\widetilde A = A + \epsilon E$ with $\epsilon >0$ sufficiently small,
    there exist bases $\widetilde X_1, \widetilde 
     Y_1 \in \mathbb C^{n\times p}$ for right/left invariant subspaces of $\widetilde A$ such that
    \begin{align*}
        \widetilde X_1   & = X_1   - \epsilon X_2 \mathbb{T}^{-1}( Y_2^*E X_1) + \mathcal{O}(\epsilon^2),\\
        \widetilde Y_1^* & = Y^*_1 - \epsilon \mathbb{T}_*^{-1}(X_2^*E^*Y_1)^*Y_2^* +  \mathcal{O}(\epsilon^2),
    \end{align*}
    where $\mathbb{T}: Z \mapsto A_{22}Z - Z A_{11}$, $\mathbb{T}_*: Z \mapsto A_{22}^*Z - Z A_{11}^*$, and $\widetilde Y_1^* \widetilde X_1 = I_p + \mathcal{O}(\epsilon^2)$.
\end{lemma} 
\begin{proof} 
   Define
    \[ F := Y^*EX = \begin{bmatrix}
        Y_1 & Y_2
    \end{bmatrix}^* E \begin{bmatrix}
        X_1 & X_2
    \end{bmatrix} = \begin{bmatrix}
        F_{11} & F_{12}\\
        F_{21} & F_{22}
    \end{bmatrix}.\]
Applying~\cite[Corollary 2.4]{KarowKressner2014} to the matrix $B := Y^*AX$ shows that there is a basis 
$\big[ \genfrac{}{}{0pt}{}{I_p}{Z_E}\big]$ 
for a right invariant subspace of $B+\epsilon F$ such that
$Z_E = -\epsilon\mathbb{T}^{-1}(F_{21}) + \Ocal{2}$.
Because the matrices $A+\epsilon E$ and $B+\epsilon F$ are similar, it follows that 
\[\widetilde X_1    =  X \begin{bmatrix}
    I_p \\ Z_E
\end{bmatrix} = X_1   - \epsilon X_2 \mathbb{T}^{-1}(Y_2^*E X_1) + \mathcal{O}(\epsilon^2)\]
is a basis for a right invariant subspace of $A+\epsilon E$.

By applying the same argument to $A^*$ we obtain the analogous result for $\widetilde Y_1$.
Because of $Y_1^*X_2 = 0$ and $Y_2^*X_1=0$, the perturbation expansions for $\widetilde X_1,\widetilde Y_1$ imply that
$\widetilde Y_1^* \widetilde X_1 = I_p + \mathcal{O}(\epsilon^2)$.
\end{proof}

The following result is a special case of Lemma~\ref{lemma:first_order_pert_subspace}, which coincides with~\cite[Theorem 2]{firstOrderPerturation} for $p=1$.

\begin{lemma}
\label{lemma:first_order_pert_cluster}
Under the notation and assumptions of Lemma~\ref{lemma:first_order_pert_subspace}, suppose that $A_{11} = \lambda I_p$, that is, 
$X_1, Y_1$ are bases for right/left eigenspaces of $A$.  
Then there exist bases $\widetilde X_1, \widetilde 
     Y_1 \in \mathbb C^{n\times p}$ for right/left invariant subspace of $\widetilde A = A+\epsilon E$ such that
    \begin{align}
        \widetilde X_1   & = X_1   - \epsilon X_2 (A_{22}-\lambda I_{n-p})^{-1} Y_2^*E X_1 + 
        \mathcal{O}\left(\epsilon^2\right),\label{eq:pert_cluster_x}\\
        \widetilde Y_1^* & = Y^*_1 - \epsilon Y_1^* E X_2 (A_{22}-\lambda I_{n-p})^{-1} Y_2^* +  \mathcal{O}\left(\epsilon^2\right), \label{eq:pert_cluster_y}
    \end{align}
    and $\widetilde Y_1^* \widetilde X_1 = I_p + \mathcal{O}(\epsilon^2)$.
\end{lemma} 

\begin{proof}
     The result follows from Lemma~\ref{lemma:first_order_pert_subspace} by noting that the operators $\mathbb{T}$ and  $\mathbb{T}_*$ from Lemma~\ref{lemma:first_order_pert_subspace} simplify to $A_{22}-\lambda I_{n-p}$ and $A^*_{22}-\overline\lambda I_{n-p}$, respectively.
\end{proof}

\end{document}